\documentclass[11pt]{article}
\usepackage{amsmath,amsthm,amssymb,amscd,fancybox,ifthen}
\usepackage[all]{xy}

\newcommand\dbarb{\bar{\partial}_b}

\DeclareMathOperator\eo{eo}
\DeclareMathOperator\ooee{oe}

\newcommand\loc{\operatorname{loc}}
\newcommand\Dom{\operatorname{Dom}}

\newcommand\eucl{\operatorname{eucl}}
\newcommand\relp{\operatorname{P}}
\newcommand\spnc{\operatorname{Spin}_{\bbC}}

\newcommand\Tr{\operatorname{tr}}
\newcommand\iso{\operatorname{iso}}
\newcommand\hn{(1,-\frac 12)}
\newcommand\cT{\mathcal T}
\newcommand\cM{\mathcal M}

\newcommand\cI{\mathcal I}

\newcommand\cU{\mathcal U}
\newcommand\cV{\mathcal V}
\newcommand\sO{\mathfrak O}
\newcommand\cQ{\mathcal Q}
\newcommand\cK{\mathcal K}

\newcommand\cB{\mathcal B}
\newcommand\cR{\mathcal R}
\newcommand\cP{\mathcal P}
\newcommand{\ccD}{\mathfrak D}
\newcommand\ho{\mathfrak H}
\newcommand\px{\tilde x}
\newcommand\pz{\tilde z}
\newcommand\bzero{\boldsymbol 0}

\newcommand\cF{\mathcal F}
\newcommand\cE{\mathcal E}

\newcommand\bX{\overline{X}}

\newcommand\dbar{\bar{\pa}}

\newcommand\bz{\bar{z}}
\newcommand\bj{\bar{j}}
\newcommand\sypr{\pi}
\newcommand\syprc{\bar{\pi}}
\newcommand\scp{\mathfrak p}
\newcommand\sd{\mathfrak d}

\newcommand\tq{\widetilde{q}}

\newcommand\tX{\widetilde{X}}
\newcommand\tK{\widetilde{K}}
\newcommand\sh[1]{\overset{\sqcap}{#1}}
\newcommand\RC[1]{{}^R\overline{T^*{#1}}}
\newcommand\HC[1]{{}^H\overline{T^*{#1}}}
\newcommand\EHC[1]{{}^{eH}\overline{T^*{#1}}}
\newcommand\sym[1]{{}^{#1}\sigma}
\newcommand\esym[2]{{}^{#1}\sigma^{#2}}

\newcommand{\Spn}{S\mspace{-10mu}/ }

\newcommand\Ker{\operatorname{ker}}

\newcommand\cA{\mathcal{A}}

\newcommand\cL{\mathcal{L}}
\newcommand\cS{\mathcal{S}}

\newcommand\bcS{\bar{\mathcal{S}}}
\newcommand\cD{\mathcal{D}}

\newcommand\tsigma{\tilde\sigma}

\newcommand\bomega{\bar{\omega}}

\newcommand\ha{\frac12}

\renewcommand\Re{\operatorname{Re}}
\renewcommand\Im{\operatorname{Im}}

\newcommand\bbC{\mathbb C}

\newcommand\bbR{\mathbb R}

\newcommand\pa{\partial}

\newcommand\restrictedto{\upharpoonright}

\newcommand\CI{{\mathcal C}^{\infty}}

\newcommand\Id{\operatorname{Id}}

\DeclareMathOperator{\odd}{o}
\DeclareMathOperator{\even}{e}

\newtheorem{theorem}{Theorem}
\newtheorem{proposition}{Proposition}
\newtheorem{corollary}{Corollary}
\newtheorem{lemma}{Lemma}

\theoremstyle{definition}

\theoremstyle{remark}
\newtheorem{remark}{Remark}

\begin{document}

\title{Subelliptic Spin\,${}_{\bbC}$ Dirac operators, II\\
Basic Estimates}

\author{Charles L. Epstein\footnote{ Research partially
supported by NSF grants DMS99-70487 and DMS02-03795 and the Francis J. Carey
term chair.  E-mail: cle@math.upenn.edu} \\ Department of Mathematics\\
University of Pennsylvania}

\date{November 2, 2007: 2nd revised version}

\maketitle

\centerline{\Large{\it This paper dedicated to Peter D. Lax}}
\centerline{\Large{\it on the occasion of his Abel Prize.}}

\begin{abstract}  We
assume that the manifold with boundary, $X,$ has a Spin${}_{\bbC}$-structure
with spinor bundle $\Spn.$ Along the boundary, this structure agrees with the
structure defined by an infinite order integrable almost complex structure and
the metric is K\"ahler. In this case the Spin${}_{\bbC}$-Dirac
operator $\eth$ agrees with $\dbar+\dbar^*$ along the boundary.  The induced
CR-structure on $bX$ is integrable and either strictly pseudoconvex or strictly
pseudoconcave. We assume that $E\to X$ is a complex vector bundle, which has an
infinite order integrable complex structure along $bX,$ compatible with that
defined along $bX.$ In this paper use boundary layer methods to prove
subelliptic estimates for the twisted Spin${}_{\bbC}$-Dirac operator acting on
sections on $\Spn\otimes E.$ We use boundary conditions that are modifications
of the classical $\dbar$-Neumann condition. These results are proved by using
the extended Heisenberg calculus.
\end{abstract}

\section*{Introduction}
Let $X$ be an even dimensional manifold with a Spin${}_{\bbC}$-structure,
 see~\cite{LawsonMichelsohn}.  A compatible choice of metric, $g,$ defines a
 Spin${}_{\bbC}$-Dirac operator, $\eth$ which acts on sections of the bundle of
 complex spinors, $\Spn.$ This bundle splits as a direct sum
 $\Spn=\Spn^{\even}\oplus\Spn^{\odd}.$ The metric on $TX$ induces a metric
 on the bundle of spinors. We let $\langle\sigma,\sigma\rangle_g$ denote the
 pointwise inner product. This, in turn, defines an inner product on the space
 of sections of $\Spn,$ by setting:
$$\langle\sigma,\sigma\rangle_X=\int\limits_{X}\langle\sigma,\sigma\rangle_g
dV_g$$

If $X$ has an almost complex structure, then this structure defines a
Spin${}_{\bbC}$-structure, see~\cite{duistermaat}.  If the complex structure is
integrable, then the bundle of complex spinors is canonically identified with
$\oplus_{q\geq 0}\Lambda^{0,q}.$ We use the notation
\begin{equation}
\Lambda^{\even}=\bigoplus\limits_{q=0}^{\lfloor\frac{n}{2}\rfloor}\Lambda^{0,2q}\quad
\Lambda^{\odd}=\bigoplus\limits_{q=0}^{\lfloor\frac{n-1}{2}\rfloor}\Lambda^{0,2q+1}.
\end{equation}
If the metric is K\"ahler, then the
Spin${}_{\bbC}$ Dirac operator is given by
$$\eth=\dbar+\dbar^*.$$ Here $\dbar^*$ denotes the formal adjoint of $\dbar$
defined by the metric. This operator is called the Dolbeault-Dirac operator by
Duistermaat, see~\cite{duistermaat}.  If the metric is Hermitian, though not
K\"ahler, then
$$\eth_{\bbC}=\dbar+\dbar^*+\cM_0,$$ 
with $\cM_0$  a homomorphism carrying $\Lambda^{\even}$ to
$\Lambda^{\odd}$ and vice versa. It vanishes at points where the metric is
K\"ahler.  It is customary to write $\eth=\eth^{\even}+\eth^{\odd}$
where
$$\eth^{\even}:\CI(X;\Spn^{\even})\longrightarrow\CI(X;\Spn^{\odd}),$$
and $\eth^{\odd}$ is the formal adjoint of $\eth^{\even}.$ 

If $X$ has a boundary, then the kernels and cokernels of $\eth^{\eo}$ are
generally infinite dimensional. To obtain a Fredholm operator we need to impose
boundary conditions. In this instance, there are no local boundary conditions
for $\eth^{\eo}$ that define elliptic problems.  Starting with the work of
Atiyah, Patodi and Singer, the basic boundary value problems for Dirac
operators on manifolds with boundary have been defined by classical
pseudodifferential projections acting on the sections of the spinor bundle
restricted to the boundary. In this paper we analyze \emph{subelliptic}
boundary conditions for $\eth^{\eo}$ obtained by modifying the classical
$\dbar$-Neumann and dual $\dbar$-Neumann conditions. The $\dbar$-Neumann
conditions on a strictly pseudoconvex manifold allow for an infinite
dimensional null space in degree $0$ and, on a strictly pseudoconcave manifold,
in degree $n-1.$ We modify these boundary conditions by using generalized
Szeg\H o projectors, in the appropriate degrees, to eliminate these infinite
dimensional spaces.

In this paper we prove the basic analytic results needed to do index theory for
these boundary value problems.  To that end, we compare the projections
defining the subelliptic boundary conditions with the Calderon projector and
show that, in a certain sense, these projections are relatively Fredholm. We
should emphasize at the outset that these projections are not relatively
Fredholm in the usual sense of say Fredholm pairs in a Hilbert space, used in
the study of elliptic boundary value problems. Nonetheless, we can use our
results to obtain a formula for a parametrix for these subelliptic boundary
value problems that is precise enough to prove, among other things, higher norm
estimates. This formula is related to earlier work of Greiner and Stein, and
Beals and Stanton, see~\cite{Greiner-Stein1,Beals-Stanton}. We use the extended
Heisenberg calculus introduced in~\cite{EpsteinMelrose3}. Similar classes of
operators were also introduced by Greiner and Stein, Beals and Stanton as well
as Taylor, see~\cite{Greiner-Stein1,Beals-Stanton,Beals-Greiner1,Taylor6}. The
results here and their applications in~\cite{Epstein4} suggest that the theory
of Fredholm pairs has an extension to subspaces of $\CI$ sections where the
relative projections satisfy appropriate tame estimates.

In this paper $X$ is a Spin${}_{\bbC}$-manifold with boundary. The
Spin${}_{\bbC}$ structure along the boundary arises from an almost complex
structure that is integrable to infinite order. This means that the induced
CR-structure on $bX$ is integrable and the Nijenhuis tensor vanishes to
infinite order along the boundary. We generally assume that this CR-structure
is either strictly pseudoconvex (or pseudoconcave). When we say that ``$X$ is a
strictly pseudoconvex (or pseudoconcave) manifold,'' this is what we
mean. We usually treat the pseudoconvex and pseudoconcave cases in tandem. When
needed, we use a subscript $+$ to denote the pseudoconvex case and $-,$ the
pseudoconcave case. 

Indeed, as all the important computations in this paper are calculations in
Taylor series along the boundary, it suffices to consider the case that the
boundary of X is in fact a hypersurface in a complex manifold, and we
often do so.  We suppose that the boundary of $X$ is the zero set of a function
$\rho$ such that
\begin{enumerate}
\item $d\rho\neq 0$ along $bX.$
\item $\partial\dbar\rho$ is positive definite along $bX.$ Hence $\rho<0,$ if
  $X$ is strictly pseudoconvex and $\rho>0,$ if $X$ is strictly pseudoconcave.
\item The length of $\dbar\rho$ in the metric with K\"ahler form $-i\partial\dbar\rho$ is
 $\sqrt{2}$ along $bX.$ This implies that the length $d\rho$ is 2 along $bX.$
\end{enumerate}
If $bX$ is a strictly pseudoconvex or pseudoconcave hypersurface, with respect
 to the infinite order integrable almost complex structure along $bX,$ then a
 defining function $\rho$ satisfying these conditions can always be found.  

The Hermitian metric on $X,$ near to $bX,$ is defined by $\pa\dbar\rho.$ If the
 almost complex structure is integrable, then this metric is K\"ahler. This
 should be contrasted to the usual situation when studying boundary value
 problems of APS type: here one usually assumes that the metric is a product in
 a neighborhood of the boundary, with the boundary a totally geodesic
 hypersurface. Since we are interested in using the subelliptic boundary value
 problems as a tool to study the complex structure of $X$ and the CR-structure
 of $bX,$ this would not be a natural hypothesis.
 Instead of taking advantage of the simplifications that arise
 from using a product metric, we use the simplifications that result from using
 K\"ahler coordinates.

Let $\cP^{\eo}$ denotes the Calderon projectors and $\cR^{\prime\eo},$ the projectors
defining the subelliptic boundary value problems on the even (odd) spinors,
respectively. These operators are defined in~\cite{Epstein4} as well as in
Lemmas~\ref{lemm4} and~\ref{lemm5}. The main objects
of study in this paper are the operators:
\begin{equation}
\cT^{\prime\eo}=\cR^{\prime\eo}\cP^{\eo}+(\Id-\cR^{\prime\eo})(\Id-\cP^{\eo}).
\label{eqn6.30.1}
\end{equation}
These operators are elements of the extended Heisenberg calculus. If $X$ is
strictly pseudoconvex, then $\cT^{\prime\eo}$ is an elliptic operator, in the
classical sense, away from the positive contact direction. Along the positive
contact direction, most of its principal symbol vanishes. If instead we compute
its principal symbol in the Heisenberg sense, we find that this symbol has a
natural block structure:
\begin{equation}
\left(\begin{matrix} A_{11} & A_{12}\\ A_{21} & A_{22}\end{matrix}\right).
\label{7.12.6}
\end{equation}
As an element of the Heisenberg calculus $A_{ij}$ is a symbol of order
$2-(i+j).$  The inverse has the identical block structure 
\begin{equation}
\left(\begin{matrix} B_{11} & B_{12}\\ B_{21} & B_{22}\end{matrix}\right),
\end{equation}
where the order of $B_{ij}$ is $(i+j)-2.$ The principal technical difficulty
that we encounter is that the symbol of $\cT^{\prime\eo}$ along the positive contact
direction could, in principle, depend on higher order terms in the symbol of
$\cP^{\eo}$ as well as the geometry of $bX$ and its embedding as the boundary
of $X.$ In fact, the Heisenberg symbol of $\cT^{\prime\eo}$ is determined by the
principal symbol of $\cP^{\eo}$ and depends in a very simple way on the
geometry of $bX\hookrightarrow X.$ It requires some effort to verify this
statement and explicitly compute the symbol. Another important result is that
the leading order part of $B_{22}$ vanishes. This allows the deduction of the
classical sharp anisotropic estimates for these modifications of the
$\dbar$-Neumann problem from our results. Analogous remarks apply to strictly
pseudoconcave manifolds with the two changes that the difficulties occur along
the negative contact direction, and the block structure depends on the parity
of the dimension.

As it entails no additional effort, we work in somewhat greater generality and
consider the ``twisted'' Spin${}_{\bbC}$ Dirac operator. To that end, we let
$E\to X$ denote a complex vector bundle and consider the Dirac operator acting
on sections of $\Spn\otimes E.$ The bundle $E$ is assumed to have an almost
complex structure near to $bX,$ that is infinite order integrable along $bX.$
We assume that this almost complex structure is compatible with that defined on
$X$ along $bX.$ By this we mean $E\to X$ defines, along $bX,$ an infinite order
germ of a holomorphic bundle over the infinite order germ of the holomorphic
manifold.We call such a bundle a complex vector bundle \emph{compatible} with
$X.$ When necessary for clarity, we let $\dbar_E$ denote the $\dbar$-operator acting on
sections of $\Lambda^{0,q}\otimes E.$ A Hermitian metric is fixed on the fibers
of $E$ and $\dbar_E^*$ denotes the adjoint operator. Along $bX,$
$\eth_E=\dbar_E+\dbar_E^*.$ In most of this paper we simplify the notation by
suppressing the dependence on $E.$

We first recall the definition of the Calderon projector in this case, which is
due to Seeley. We follow the discussion in~\cite{BBW}. We then examine its
symbol and the symbol of $\cT^{\eo}_{\pm}$ away from the contact directions. Next we
compute the symbol in the appropriate contact direction. We see that
$\cT^{\eo}_{\pm}$ is a graded elliptic system in the extended Heisenberg
calculus. Using the parametrix for $\cT^{\eo}_{\pm}$ we obtain parametrices for the
boundary value problems considered here as well as those introduced
in~\cite{Epstein4}. Using the parametrices we prove subelliptic estimates for
solutions of these boundary value problems formally identical to the classical
$\dbar$-Neumann estimates of Kohn. We are also able to characterize the adjoints of the
graph closures of the various operators as the graph closures of the formal
adjoints. 

{\small \centerline{Acknowledgments} Boundary conditions similar to those
considered in this paper were first suggested to me by Laszlo Lempert. I would
like to thank John Roe for some helpful pointers on the Spin${}_{\bbC}$ Dirac
operator.}

\section{The extended Heisenberg Calculus}
The main results in this paper rely on the computation of the symbol of an
operator built out of the Calderon projector and a projection operator in the
Heisenberg calculus. This operator belongs to the extended Heisenberg calculus,
as defined in~\cite{EpsteinMelrose3}. While we do not intend to review this
construction in detail, we briefly describe the different symbol classes within
a single fiber of the cotangent bundle. This suffices for our purposes as all
of our symbolic computations are principal symbol computations, which are, in
all cases, localized to a single fiber.

Each symbol class is defined by a compactification of the fibers of $T^*Y.$ In
our applications, $Y$ is a contact manifold of dimension $2n-1.$ Let $L$ denote
the contact line within $T^*Y.$ We assume that $L$ is oriented and $\theta$ is
a global, positive section of $L.$ According to Darboux's theorem, there are
coordinates $(y_0,y_1,\dots,y_{2(n-1)})$ for a neighborhood $U$ of $p\in Y,$
so that
\begin{equation}
\theta\restrictedto_{U}=dy_0+\frac{1}{2}\sum_{j=1}^{n-1}[y_{j}dy_{j+n}-y_{j+n}dy_{j}],
\label{7.14.3}
\end{equation}
Let $\eta$ denote the local fiber coordinates on $T^*Y$ defined by the
trivialization
$$\{dy_0,\dots,dy_{2(n-1)}\}.$$ We often use the splitting
$\eta=(\eta_0,\eta').$ In the remainder of this section we do essentially all
our calculations at the point $p.$ As such coordinates can be found in a
neighborhood of any point, and in light of the invariance results established
in~\cite{EpsteinMelrose3}, these computations actually cover the general case.

\subsection{The compactifications of $T^*Y$}
We define three
compactifications of the fibers of $T^*Y.$ The first is the standard radial
compactification, $\RC{Y},$ defined by adding one point at infinity for each
orbit of the standard $\bbR_+$-action, $(y,\eta)\mapsto (y,\lambda\eta).$ 
Along with $y$, standard polar coordinates in the $\eta$-variables define local
coordinates near $b\RC{Y}:$ 
\begin{equation}
r_R=\frac{1}{|\eta|},\quad \omega_j=\frac{\eta_j}{|\eta|},
\label{7.7.4}
\end{equation}
with $r_R$ a smooth defining function for $b\RC{Y}.$

To define the Heisenberg compactification we first need to define a parabolic
action of $\bbR_+.$ Let $T$ denote the vector field defined by the conditions
$\theta(T)=1,i_Td\theta=0.$ As usual $i_T$ denotes interior product with the
vector field $T.$ Let $H^*$ denote the subbundle of $T^*Y$ consisting of forms
that annihilate $T.$ Clearly $T^*Y=L\oplus H^*,$ let $\pi_L\oplus\pi_{H^*}$
denote the bundle projections defined by this splitting. The parabolic action
of $\bbR_+$ is defined by
\begin{equation}
(y,\eta)\mapsto (y,\lambda \pi_{H^*}(y,\eta)+\lambda^2\pi_{T}(y,\eta))
\end{equation}
In the Heisenberg compactification we add one point at infinity for each orbit
under this action. A smooth defining function for the boundary is given by
\begin{equation}
r_H=[|\pi_{H^*}(y,\eta)|^4+|\pi_{T}(y,\eta)|^2]^{-\frac 14}.
\end{equation}
In~\cite{EpsteinMelrose3} it is shown that the smooth structure of $\HC{Y}$
depends only on the contact structure, and not the choice of contact form.

In the fiber over $y=0,$  $r_H=[|\eta'|^4+|\eta_0|^2]^{-\frac 14}.$
Coordinates near the boundary in the fiber over $y=0$ are given by
\begin{equation}
r_H, \sigma_0=\frac{\eta_0}{[|\eta'|^4+|\eta_0|^2]^{\frac 12}},
\sigma_j=\frac{\eta_j}{[|\eta'|^4+|\eta_0|^2]^{\frac 14}},\quad j=1,\dots,
2(n-1).
\end{equation}

The extended Heisenberg compactification can be defined by performing a blowup
of either the radial or the Heisenberg compactification. Since we need to lift
classical symbols to the extended Heisenberg compactification, we describe the
fiber of $\EHC{Y}$ in terms of a blowup of $\RC{Y}.$ In this model we
parabolically blowup the boundary of contact line, i.e., the boundary of the
closure of $L$ in $\RC{Y}.$ The conormal bundle to the $b\RC{Y}$ defines the
parabolic direction. The fiber of the compactified space is a manifold with
corners, having three hypersurface boundary components.  The two boundary
points of $\overline{L}$ become $2(n-1)$ dimensional disks. These are called
the upper and lower Heisenberg faces. The complement of $b\overline{L}$ lifts
to a cylinder, diffeomorphic to $(-1,1)\times S^{2n-3},$ which was call the
``classical'' face. Let $r_{e\pm}$ be defining functions for the upper and
lower Heisenberg faces and $r_c$ a defining function for the classical face.
From the definition we see that coordinates near the Heisenberg faces, in the
fiber over $y=0,$ are given by
\begin{equation}
r_{eH}=[r_R^2+|\omega'|^4]^{\frac 14},
\quad\tsigma_j=\frac{\omega_j}{r_{eH}},\text{ for } j=1,\dots,2n-2,
\label{7.7.5}
\end{equation}
with $r_{eH}$ a smooth defining function for the Heisenberg faces. In order for
an arc within $T^*Y$ to approach either Heisenberg face it is necessary that,
for any $\epsilon>0,$
$$|\eta'|\leq \epsilon |\eta_0|,$$ 
as $|\eta|$ tends to infinity. Indeed, for arcs that terminate on the interior
of a Heisenberg face the ratio $\eta'/\sqrt{|\eta_0|}$ approaches a limit. If
$\eta_0\to +\infty\,(-\infty),$ then the arc approaches the upper (lower)
parabolic face.  In the interior of the Heisenberg faces we can use
$[|\eta_0|]^{-\frac 12}$ as a defining function.

\subsection{The symbol classes and pseudodifferential operators}
The  symbols of order zero are defined in all cases as the smooth
functions on the compactified cotangent space:
\begin{equation}
S^0_{R}=\CI(\RC{Y}),\quad S^0_{H}=\CI(\HC{Y}),\quad S^0_{eH}=\CI(\EHC{Y}).
\end{equation}
In the classical and Heisenberg cases there is a single order parameter for
symbols, the symbols of order $m$ are defined as
\begin{equation}
S^m_{R}=r_R^{-m}\CI(\RC{Y}),\quad S^m_{H}=r_H^{-m}\CI(\HC{Y}).
\end{equation}
In the extended Heisenberg case there are three symbolic orders $(m_c,m_+,m_-),$
the symbol classes are defined by
\begin{equation}
S^{m_c,m_+,m_-}_{eH}=r_c^{-m_c}r_{e+}^{-m_+}r_{e-}^{-m_-}S^0_{eH}.
\end{equation}
If $a$ is a symbol belonging to one of the three classes above, and
$\varphi$ is a smooth function with compact support in $U,$ then the
Weyl quantization rule is used to define the localized operator $M_{\varphi}a(X,D)
M_{\varphi}:$ 
\begin{equation}
M_\varphi a(X,D)M_\varphi f=
\int\limits_{R^{2n-1}}\int\limits_{R^{2n-1}}
\varphi(y)a(\frac{y+y'}{2},\eta)\varphi(y')f(y')e^{i\langle\eta,y-y'\rangle}\frac{dy'd\eta}
{(2\pi)^{2n-1}}.
\end{equation}
The operator $M_{\varphi}$ is multiplication by $\varphi.$ As usual, the Schwartz
kernel of $a(X,D)$ is assumed to be smooth away from the diagonal.
    
We denote the classes of pseudodifferential operators defined by the symbol
classes $S_{R}^m, S_{H}^m, S_{eH}^{m_c,m_+,m_-}$ by
$\Psi_{R}^m,\Psi_{H}^m,\Psi_{eH}^{m_c,m_+,m_-},$ respectively. As usual, the
leading term in the Taylor expansion of a symbol along the boundary can be used
to define a principal symbol. Because the defining functions for the boundary
components are only determined up to multiplication by a positive function,
invariantly, these symbols are sections of line bundles defined on the
boundary. We let $\sym{R}_m(A),$ $\sym{H}_m(A)$ denote the principal symbols
for the classical and Heisenberg pseudodifferential operators of order $m.$ In
each of these cases, the principal symbol uniquely determines a function on
the cotangent space, homogeneous with respect to the appropriate $\bbR_+$
action.  An extended Heisenberg operator has three principal symbols,
corresponding to the three boundary hypersurfaces of $\EHC{Y}.$ For an operator
with orders $(m_c,m_+,m_-)$ they are denoted by $\esym{eH}{c}_{m_c}(A),
\esym{eH}{}_{m_+}(+)(A),$ $\esym{eH}{}_{m_-}(-)(A).$ The classical symbol
$\esym{eH}{c}_{m_c}(A)$ can be represented by a radially homogeneous function
defined on $T^*Y\setminus L.$ The vector field $T$ defines a splitting to
$T^*Y$ into two half spaces
\begin{equation}
T^*_{\pm}Y=\{(y,\eta):\: \pm\eta(T)>0\}.
\end{equation}
The Heisenberg symbols, $\esym{eH}{}_{m_{\pm}}(\pm)(A)$ can be represented by
parabolically homogeneous functions defined in the half spaces of
$T^*_{\pm}Y.$ In most of our computations we use the representations of
principal symbols in terms of functions, homogeneous with respect to the
appropriate $\bbR_+$-action.

\subsection{Symbolic composition formul\ae\ }
The quantization rule leads to a different symbolic composition rule for each
class of operators. For classical operators, the composition of principal
symbols is given by pointwise multiplication: If $A\in\Psi_R^m, B\in\Psi_R^{m'},$
then $A\circ B\in\Psi_R^{m+m'}$ and
\begin{equation}
\sym{R}_{m+m'}(A\circ B)(p,\eta)=\sym{R}_m(A)(p,\eta)\sym{R}_{m'}(B)(p,\eta).
\label{7.7.1}
\end{equation}
For Heisenberg operators, the composition rule involves a nonlocal operation in
the fiber of the cotangent space. If $A\in\Psi_H^m, B\in\Psi_H^{m'},$ then
$A\circ B\in\Psi_H^{m+m'}.$ For our purposes it suffices to give a formula for
$\sym{H}_{m+m'}(A\circ B)(p,\pm 1,\eta');$ the symbol is then extended to
$T_p^*Y\setminus H^*$ as a parabolically homogeneous function of degree $m+m'.$
It extends to $H^*\setminus\{0\}$ by continuity. On the hyperplanes $\eta_0=\pm
1$ the composite symbol is given by
\begin{equation}
\begin{split}
\sym{H}_{m+m'}&(A\circ B)(p,\pm 1,\eta')=\\
&\frac{1}{\pi^{2(n-1)}}\int\limits_{\bbR^{2(n-1)}}\int\limits_{\bbR^{2(n-1)}}
a_m(\pm 1,u+\eta')b_{m'}(\pm 1,v+\eta')
e^{\pm2i\omega(u,v)}dudv,
\end{split}
\label{7.7.2}
\end{equation}
where $\omega=d\theta',$ the dual of $d\theta\restrictedto_{H^*},$ and
$$a_m(\eta)=\sym{H}_m(A)(p,\eta),\quad b_{m'}(\eta)=\sym{H}_{m'}(B)(p,\eta).$$
Note that the composed symbols in each half space are determined by the
component symbols in that half space. Indeed the symbols that vanish in a half
space define an ideal. These are called the upper and lower Hermite ideals.
The right hand side of~\eqref{7.7.2} defines two associative products on
appropriate classes of functions defined on $\bbR^{2(n-1)},$ which are
sometimes denoted by $a_m\sharp_{\pm} b_{m'}.$ An operator in $\Psi_H^m$ is
elliptic if and only if the functions $\sym{H}_{m}(p,\pm 1,\eta')$ are
invertible elements, or units, with respect to these algebra structures.

Using the representations of symbols as homogeneous functions, the compositions
for the different types of extended Heisenberg symbols are defined using the
appropriate formula above: the classical symbols are composed
using~\eqref{7.7.1} and the Heisenberg symbols are composed
using~\eqref{7.7.2}, with $+$ for $\esym{eH}{}(+)$ and $-$ for $\esym{eH}{}(-).$
These formul\ae\ and their invariance properties are established
in~\cite{EpsteinMelrose3}. 

The formula in~\eqref{7.7.2} would be of little use, but for the fact that it
has an interpretation as a composition formula for a class of operators acting
on $\bbR^{n-1}.$  The restrictions of a Heisenberg symbol to the hyperplanes
$\eta_0=\pm 1$ define \emph{isotropic} symbols on $\bbR^{2(n-1)}.$ An isotropic
symbol is a smooth function on $\bbR^{2(n-1)}$ that satisfies symbolic
estimates in all variables, i.e., $c(\eta')$ is an isotropic symbol of order
$m$ if, for every $2(n-1)$-multi-index $\alpha,$ there
is a constant $C_{\alpha}$ so that
\begin{equation}
|\pa_{\eta'}^\alpha c(\eta')|\leq C_{\alpha}(1+|\eta'|)^{m-|\alpha|}.
\end{equation}
We split $\eta'$ into two parts
\begin{equation}
x = (\eta_1,\dots,\eta_{n-1}),\quad \xi =(\eta_{n},\dots,\eta_{2(n-1)}).
\label{eqn19}
\end{equation}
If $c$ is an isotropic symbol, then we define two operators acting on
$\cS(\bbR^{n-1})$ by defining the Schwartz kernels of $c^{\pm}(X,D)$ to be
\begin{equation}
k_c^{\pm}(x,x')=\int\limits_{\bbR^{n-1}}
e^{\pm i\langle\xi,x-x'\rangle}c(\frac{x+x'}{2},\xi)d\xi.
\label{7.7.3}
\end{equation}
The utility of the formula in~\eqref{7.7.2} is a consequence of the following
proposition:
\begin{proposition} If $c_1$ and $c_2$ are two isotropic symbols,  then the
  complete symbol of $c_1^{\pm}(X,D)\circ c_2^{\pm}(X,D)$ is $c_1\sharp_{\pm} c_2,$
  with $\omega=\sum dx_j\wedge d\xi_j.$ 
  An isotropic operator $c^{\pm}(X,D):\cS(\bbR^{n-1})\to\cS(\bbR^{n-1})$ is
  invertible if and only if $c(\eta')$ is a unit with respect to the
  $\sharp_{\pm}$ product.
\end{proposition}
\begin{remark} This result appears in essentially this form
  in~\cite{Taylor3}. It is related to an earlier result of Rockland.
\end{remark}

If $A$ is a Heisenberg (or extended Heisenberg operator), then the isotropic
symbols $\sym{H}_m(A)(p,\pm 1,\eta')$ ($\esym{eH}{\pm}(A)(p,\pm 1,\eta')$) can
be quantized using~\eqref{7.7.3}. We denote the corresponding operators by
$\sym{H}_m(A)(p,\pm),$ ($\esym{eH}{}(A)(p,\pm)$). We call these ``the'' model
operators defined by $A$ at $p.$ Often the point of evaluation, $p$ is fixed
and then it is omitted from the notation. The choice of splitting
in~\eqref{eqn19} cannot in general be done globally. Hence the model operators
are not, in general, globally defined. What is important to note is that the
invertibility of these operators does \emph{not} depend on the choices made to define
them. From the proposition it is clear that $A$ is elliptic in the Heisenberg
calculus if and only if the model operators are everywhere invertible. An
operator in the extended Heisenberg calculus is elliptic if and only if these
model operators are invertible and the classical principal symbol is
nonvanishing.

All these classes of operators are easily extended to act between sections of
vector bundles. When necessary we indicate this by using,
e.g. $\Psi^m_R(Y;F_1,F_2)$ to denote classical pseudodifferential operators of
order $m$ acting from sections of the bundle $F_1$ to sections of the bundle
$F_2.$ In this case the symbols take values in $P^*(\hom(F_1,F_2)),$ where
$P:T^*Y\to Y$ is the canonical projection. Unless needed for clarity, the
explicit dependence on bundles is suppressed.

\subsection{Lifting classical symbols to $\EHC{Y}$}
We close our discussion of the extended Heisenberg calculus by considering
lifts of classical symbols from $\RC{Y}$ to $\EHC{Y}.$ As above, it suffices to
consider what happens on the fiber over $p.$ This fixed point of evaluation is
suppressed to simplify the notation. Let $a(\eta)$ be a classically
homogeneous function of degree $m.$ The transition from the radial
compactification to the extended Heisenberg compactification involves blowing
up the points $(\pm\infty,0)$ in the fiber of $\RC{Y}.$ We need to understand
the behavior of $a$ near these points. Away from $\eta=0,$ we can express
$a(\eta)=r_{R}^{-m} a_0(\omega),$ where $a_0$ is a homogeneous function of
degree $0.$ Using the relations in~\eqref{7.7.4} and~\eqref{7.7.5} we see that
\begin{equation}
r_R=r_{eH}^2\sqrt{1-|\tsigma'|^4},\quad
\omega'=r_{eH}\tsigma'.
\label{7.7.7}
\end{equation}
Near $b\overline{L}$ we can use $r_R$ and $\omega'$ as coordinates, where the
function $a$ has  Taylor expansions:
\begin{equation}
a_{\pm}(r_R,\omega')=r_R^{-m}a_0(\pm\sqrt{1-|\omega'|^2},\omega')
\sim r_R^{-m}\sum\limits_{\alpha}
a^{(\alpha)}_{\pm}\omega^{\prime\alpha}.
\label{7.7.6}
\end{equation}
To find the lift, we substitute from~\eqref{7.7.7} into~\eqref{7.7.6} to obtain
\begin{equation}
a(r_{eH},\tsigma')\sim r_{eH}^{-2m}(1-|\tsigma'|^4)^{-\frac m2}\sum\limits_{\alpha}
a^{(\alpha)}_{\pm}r_{eH}^{|\alpha|}\tsigma^{\prime\alpha}.
\label{7.7.8}
\end{equation}
We summarize these computations in a proposition.
\begin{proposition} Let $a(\eta)$ be a classically homogeneous function of
  order $m$ with Taylor expansion given
in~\eqref{7.7.6}. If $a^{(\alpha)}_{\pm}$ vanish for $|\alpha|< k_{\pm},$ then
the symbol $a\in S_{R}^m$ lifts to define an element of
$S_{eH}^{m,2m-k_+,2m-k_-}.$ The Heisenberg principal symbols (as  sections of
line bundles on the boundary) are given by
\begin{equation}
{}^{eH}a_{m_{\pm}}=r^{k_{\pm}-2m}_{eH}(1-|\tsigma'|^4)^{-\frac
  m2}\sum\limits_{|\alpha|=k_{\pm}} a_{\pm}^{(\alpha)}\tsigma^{\prime\alpha}.
\label{7.7.9}
\end{equation}
\end{proposition}
\begin{remark} From this proposition it is clear that the Heisenberg principal
  symbol of the lift of a classical pseudodifferential operator may not be
  defined by its classical principal symbol. It may depend on lower order terms in
  the classical symbol.
\end{remark}

To compute with the lifted symbols it is more useful to represent them as
Heisenberg homogeneous functions. In the computations that follow we only
encounter symbols of the form
\begin{equation}
a(\eta)=\frac{h(\eta)}{|\eta|^k},
\end{equation}
with $h(\eta)$ a polynomial of degree $l.$ In the fiber over $p,$ the
coordinate $\eta_0$ is parabolically homogeneous of degree $2$ whereas the
coordinates in $\eta'$ are parabolically homogeneous of degree $1.$ Using this
observation, it is straightforward to find the representations, as
parabolically homogeneous functions, of the Heisenberg principal symbols
defined by $a(\eta).$ First observe that $|\eta'|^2/\eta_0$ is parabolically
homogeneous of degree $0,$ and therefore, in terms of the \emph{parabolic
homogeneities} we have the expansion
\begin{equation}
\begin{split}
\frac{1}{|\eta|^k}&=\frac{1}{|\eta_0|^k}\frac{1}
{\left(1+\frac{|\eta'|^2}{\eta_0^2}\right)^k}\\
&\sim\frac{1}{|\eta_0|^k}\left[1+\sum\limits_{j=1}^\infty
  \frac{c_{k,j}}{|\eta_0|^j}\left(\frac{|\eta'|^2}{|\eta_0|}\right)^j
\right].
\end{split}
\label{7.8.1}
\end{equation}
Thus $|\eta|^{-k}$ lifts to define a symbol in $S^{-k,-2k,-2k}_{eH}.$ Note also
that  only even parabolic degrees appear in this expansion. 

We complete the analysis by expressing $h(\eta)$ as a polynomial in $\eta_0:$
\begin{equation}
h(\eta)=\sum_{j=0}^{l'}\eta_0^j h_j(\eta'),
\label{7.8.2}
\end{equation}
here $h_j$ is a radially homogeneous polynomial of degree $l-j,$ and $l'\leq
l.$ We assume that $h_{l'}\neq 0.$ Evidently $\eta_0^{l'}h_{l'}(\eta')$ is the
term with highest parabolic order, and therefore $h$ lifts to define a
parabolic symbol of order $l'+l.$ Combining these calculations gives the
following result:
\begin{proposition}\label{prop33}
If $h(\eta)$ is a radially homogeneous polynomial of degree $l$ with expansion
given by~\eqref{7.8.2}, then $h(\eta)|\eta|^{-k}$ lifts to define an element of
$S^{l-k,l'+l-2k,l'+l-2k}_{eH}.$ As parabolically homogeneous functions, the
Heisenberg principal symbols are
\begin{equation}
(\pm 1)^{l'}|\eta_0|^{l'-k}h_{l'}(\eta').
\label{7.8.3}
\end{equation}
\end{proposition}
\begin{proof} The statement about the orders of the lifted symbols follows
  immediately from~\eqref{7.8.1} and ~\eqref{7.8.2}. We observe that
  $|\eta_0|^{-\frac 12}$ is a defining function for the upper and lower
  Heisenberg faces, and $\eta'/\sqrt{|\eta_0|}$ is parabolically homogeneous of
  degree $0.$ As noted, the term in the expansion of $h(\eta)|\eta|^{-k}$ with highest
  parabolic degree is that given in~\eqref{7.8.3}. We can express it as the leading term
  in the Taylor series of the lifted symbol along the Heisenberg face as:
\begin{equation}
(\pm 1)^{l'}|\eta_0|^{l'-k}h_{l'}(\eta')=(\pm 1)^{l'}[\sqrt{|\eta_0|}]^{l+l'-2k}
h_{l'}\left(\frac{\eta'}{\sqrt{|\eta_0|}}\right).
\label{7.8.4}
\end{equation}
\end{proof}
Note that the terms in the parabolic expansions of the lift of
$h(\eta)|\eta|^{-k}$ all have the same parity.

\section{The symbol of the Dirac Operator and its inverse}\label{ss.2}
Let $X$ be a manifold with boundary, $Y$ and suppose that $X$ has a
$\spnc$-structure and a compatible metric. Let $\eth_E$ denote the twisted
$\spnc$-Dirac operator and $\eth^{\eo}_E$ its ``even'' and ``odd'' parts. Let
$\rho$ be a defining function for $bX.$ As noted above, $E\to X$ is a complex
vector bundle with compatible almost complex structure along $bX.$ The manifold
$X$ can be included into a larger manifold $\tX$ in such a way that its
$\spnc$-structure and Dirac operator extend smoothly to $\tX$ and such that the
operators $\eth^{\eo}_E$ are invertible, see Chapter 9 of~\cite{BBW}. Let
$Q^{\eo}_E$ denote the inverses of $\eth^{\eo}_E.$ These are classical
pseudodifferential operators of order $-1.$ The existence of an exact inverse
just simplifies the presentation a little, a parametrix suffices for our
computations.

Let $r$ denote the operation of restriction of a section of
$\Spn^{\eo}\otimes E,$ defined on $\tX$ to $X,$ and $\gamma_\epsilon$ the operation of
restriction of a smooth section of $\Spn^{\eo}\otimes E$ to
$Y_{\epsilon}=\{\rho^{-1}(\epsilon)\}.$ We use the convention used
in~\cite{Epstein4}: if $X$ is strictly pseudoconvex then $\rho<0$ on $X$ and if
$X$ is strictly pseudoconcave then $\rho>0$ on $X.$ We define the 
operator
\begin{equation}
\tK^{\eo}_E\overset{d}{=} r Q^{\eo}_E\gamma^*_0:\CI(Y;\Spn^{\ooee}\otimes E\restrictedto_{Y})
\longrightarrow \CI(X;\Spn^{\eo}\otimes E).
\end{equation}
Here $\gamma^*_0$ is the formal adjoint of $\gamma_0.$ We recall that, along
$Y$ the symbol $\sigma_1(\eth^{\eo}_E,d\rho)$ defines an isomorphism
\begin{equation}
\sigma_1(\eth^{\eo},d\rho):\Spn^{\eo}\otimes E\restrictedto_{Y}\longrightarrow
\Spn^{\ooee}\otimes E\restrictedto_{Y}.
\end{equation}
Composing, we get the usual Poisson operators
\begin{equation}
\cK^{\eo}_{E\pm}=\frac{\mp}{i\sqrt{2}}\tK^{\eo}_E\circ\sigma_1(\eth^{\eo}_E,d\rho):
\CI(Y;\Spn^{\eo}\otimes E\restrictedto_{Y})
\longrightarrow \CI(X;\Spn^{\eo}\otimes E),
\end{equation}
which map sections of $\Spn^{\eo}\otimes E\restrictedto_{Y}$ into the nullspace of
$\eth^{\eo}_E.$ The factor $\mp/\sqrt{2}$ is inserted because $\rho<0$ on $X,$ if
$X$ is strictly pseudoconvex, and $\|d\rho\|=\sqrt{2}.$

The Calderon projectors are defined by
\begin{equation}
\cP^{\eo}_{E\pm} s\overset{d}{=}\lim_{\mp\epsilon\to 0^+}
\gamma_{\epsilon}\cK^{\eo}_{E\pm}s\text{ for }s\in\CI(Y;\Spn^{\eo}\otimes E\restrictedto_{Y}).
\end{equation}
The fundamental result of Seeley is that $\cP^{\eo}_{E\pm}$ are classical
pseudodifferential operators of order $0.$ The ranges of these operators are the
boundary values of elements of $\Ker\eth^{\eo}_{E\pm}.$ Seeley gave a
prescription for computing the symbols of these operators using contour
integrals, which we do not repeat, as we shall be computing these symbols in
detail in the following sections. See~\cite{Seeley0} 

\begin{remark}[{\bf Notational remark}] The notation $\cP_{\pm}$ used in
  this paper does \emph{not} follow the usual convention in this
  field. Usually $\cP_{\pm}$ would refer to the Calderon projectors defined
  by approaching a hypersurface in a single invertible double from either
  side. In this case one would have the identity $\cP_+ +\cP_-=\Id.$ In our usage,
  $\cP_+$ refers to the projector for the pseudoconvex side and $\cP_{-}$ the
  projector for the pseudoconcave side. With our convention it is not usually
  true that $\cP_++\cP_-=\Id.$
\end{remark}

As we need to compute the symbol of $Q^{\eo}_E$ is some detail, we now consider
how to find it. We start with the formally self adjoint operators
$D^{\eo}_E=\eth^{\eo}_E\eth^{\ooee}_E.$ If $Q_{E(2)}^{\eo}$ is the inverse of $D^{\eo}_E$
then
\begin{equation}
Q^{\eo}_E=\eth^{\ooee}_E Q_{E(2)}^{\eo}.
\end{equation}
In carefully chosen coordinates, it is a simple matter to get a precise
description of symbols of $\eth^{\eo}_E$ and $Q_{E(2)}^{\eo}$ and thereby the
symbols of $Q^{\eo}_E.$ Throughout this and the following section we
repeatedly use the fact that the principal symbol of a classical, Heisenberg or
extended Heisenberg pseudodifferential operator is well defined as a
(collection of) homogeneous functions on the cotangent bundle. To make these
computations tractable it is crucial to carefully normalize the coordinates.
At the boundary, there is a complex interplay between the K\"ahler geometry of
$X$ and the CR-geometry of $bX.$ For this reason the initial computations are
done in a K\"ahler coordinate system about a fixed point $p\in bX.$ In order to
compute the symbol of the Calderon projector we need to switch to a boundary
adapted coordinate system. Finally, to analyze the Heisenberg symbols of
$\cT^{\eo}_{E\pm}$ we need to use Darboux coordinates at $p.$ Since the boundary
is assumed to be strictly pseudoconvex (pseudoconcave), the relevant geometry
is the same at every boundary point, hence there is no loss in generality in doing
the computations at a fixed point.

We now suppose that, in a neighborhood of the boundary, $X$ is a complex
manifold and the K\"ahler form of the metric is given by
$\omega_g=-i\pa\dbar\rho.$ We are implicitly assuming that $bX$ is either
strictly pseudoconvex or strictly pseudoconcave. Our convention on the sign of
$\rho$ implies that, in either case, $\omega_g$ is positive definite near to
$bX.$ As noted above it is really sufficient to assume that $X$ has an almost
complex structure along $bX$ that is integrable to infinite order, however, to
simplify the exposition we assume that there is a genuine complex structure in
a neighborhood of $bX.$ We fix an Hermitian metric $h$ on sections of $E.$

Fix a point $p$ on the boundary of $X$ and let $(z_1,\dots,z_n)$ denote
K\"ahler coordinates centered at $p.$ This means that
\begin{enumerate}
\item $p\leftrightarrow (0,\dots,0)$
\item The Hermitian metric tensor $g_{i\bj}$ in these coordinates satisfies
\begin{equation}
g_{i\bj}=\frac{1}{2}\delta_{i\bj}+O(|z|^2).
\label{7.2.1}
\end{equation}
\end{enumerate}
As a consequence of Lemma 2.3 in~\cite{wells}, we can choose a local
holomorphic frame $(e_1(e),\dots,e_r(z))$ for $E$ such that
\begin{equation}
h(e_j(z),e_k(z))=\delta_{jk}+O(|z|^2).
\label{7.26.1}
\end{equation}
Equation~\eqref{7.2.1} implies that, after a linear change of coordinates,
  we can arrange to have
\begin{equation}
\rho(z)=-2\Re z_1+|z|^2+\Re (bz,z)+O(|z|^3).
\label{7.2.2}
\end{equation}
In this equation $b$ is an $n\times n$ complex matrix and
\begin{equation}
(w,z)=\sum_{j=1}^n w_jz_j.
\end{equation}
We use the conventions for K\"ahler geometry laid out in Section IX.5
of~\cite{KobayashiNomizu2}.  The underlying real coordinates are denoted by
$(x_1,\dots,x_{2n}),$ with $z_j=x_j+ix_{j+n},$ and $(\xi_1,\dots,\xi_{2n})$
denote the linear coordinates defined on the fibers of $T^*X$ by the local
coframe field $\{dx_1,\dots,dx_{2n}\}.$

In this coordinate system we now compute the symbols of $\eth_E=\dbar_E+\dbar^*_E,$
$D^{\eo}_E,$ $Q^{\eo}_{E(2)}$ and $Q^{\eo}_E.$ For these calculations the following
notation proves very useful: a term which is a symbol of order at most $k$ vanishing,
at $p,$ to order $l$ is denoted by $\sO_{k}(|z|^l).$ As we work with a variety
of operator calculi, it is sometimes necessary to be specific as to the sense
in which the order should be taken. The notation $\sO^{C}_j$ refers to terms of
order at most $j$ in the sense of the class $C.$ If $C=eH$ we sometimes use an
appropriate multi-order. If no symbol class is specified, then the order  is with
respect to the classical, radial scaling. If no rate of vanishing is specified,
it should be understood to be $O(1).$

Recall that, with
respect to the standard Euclidean metric 
\begin{equation}
\langle\pa_{\bz_j},\pa_{\bz_k}\rangle_{\eucl}=\frac{1}{2}\text{ and }
\langle d\bz_j,d\bz_k\rangle_{\eucl}=2.
\end{equation}
Orthonormal bases for $T^{1,0}X$ and $\Lambda^{1,0}X,$ near to $p,$ take the form
\begin{equation}
Z_j=\sqrt{2}(\pa_{z_j}+e_{jk}(z)\pa_{z_k}),\quad
\omega_j=\frac{1}{\sqrt{2}}(dz_j+f_{jk}(z)dz_k),
\end{equation}
with $e_{jk}$ and $f_{jk}$ both $O(|z|^2).$ With respect to the
trivialization of $E$ given above, the symbol of $\eth_E$ is a
polynomial in $\xi$ of the form
\begin{equation}
\sigma(\eth_E)(z,\xi)=d(z,\xi)=d_1(z,\xi)+d_0(z),
\end{equation}
with $d_j(z,\cdot)$ a polynomial of degree $j$ such that
\begin{equation}
d_1(z,\xi)=d_1(0,\xi)+\sO_1(|z|^2),\quad
d_0(z)=\sO_0(|z|).
\label{7.2.3}
\end{equation}
The linear polynomial $d_1(0,\xi)$ is the symbol of $\dbar_E+\dbar^*_E$ on $\bbC^n$
with respect to the flat metric. These formul\ae\ imply that
\begin{equation}
\sigma(D^{\eo}_E)=\Delta_2(z,\xi)+\Delta_1(z,\xi)+\Delta_0(z,\xi),
\end{equation}
with $\Delta_j(z,\cdot)$ a polynomial of degree $j$ such that
\begin{equation}
\begin{split}
\Delta_2(z,\xi)=\Delta_2&(0,\xi)+\sO_2(|z|^2)\\
\Delta_1(z,\xi)=\sO_1(|z|), &\quad \Delta_0(z,\xi)=\sO_0(1).
\end{split}
\label{7.2.4}
\end{equation}
As the metric is K\"ahler, $D^{\eo}_E$ is half the Riemannian Laplacian, hence
the principal symbol at zero is
\begin{equation}
\Delta_2(0,\xi)=\frac{1}{2}|\xi|^2\otimes\Id.
\end{equation}
Here $\Id$ is the identity homomorphism on the appropriate bundle. As it has no
significant effect on our subsequence computations, or results, we heretofore
suppress the explicit dependence on the bundle $E,$ except where necessary.

The symbol $\sigma(Q^{\eo}_{(2)})=\tq=\tq_{-2}+\tq_{-3}+\dots$ is determined by the usual
symbolic relations:
\begin{equation}
\begin{split}
\tq_{-2}&=\Delta_2^{-1}\\
\tq_{-3}=-\tq_{-2}[\Delta_1\tq_{-2}&+iD_{\xi_j}\Delta_2D_{x_j}\tq_{-2}],
\end{split}
\label{7.2.5}
\end{equation}
\emph{etc}. Using the expressions in~\eqref{7.2.4} we obtain that
\begin{equation}
\begin{split}
\tq_{-2}&=\frac{2}{|\xi|^2}(\Id+\sO_0(|z|^2))\\
\tq_{-3}&=\frac{\sO_1(|z|)}{|\xi|^4},
\end{split}
\label{7.2.6}
\end{equation}
and generally for $k\geq 2$ we have
\begin{equation}
\begin{split}
\tq_{-2k}&=\sum_{j=0}^{l_k}\frac{\sO_{2j}(1)}{|\xi|^{2(k+j)}}\\
\tq_{-(2k+1)}&=\sum_{j=0}^{l'_k}\frac{\sO_{1+2j}(1)}{|\xi|^{2(k+j+1)}}.
\end{split}
\label{7.2.7}
\end{equation}
The exact form of denominator is important in the computation of the symbol of
Calderon projectors.  The numerators  are polynomials in $\xi$ of the indicated
degrees. 

Set
\begin{equation}
\sigma(Q^{\eo})=q=q_{-1}+q_{-2}+\dots
\end{equation}
As it has no bearing on the calculation, for the moment we do not keep track of
whether to use the even or odd part of the operator.  Note that the symbol of
$Q^{\eo}_{(2)}$ is the same for both parities. From the standard composition
formula, we obtain that
\begin{equation}
\begin{split}
q_{-1}&=d_1\tq_{-2}\\
q_{-2}&=d_1\tq_{-3}+d_0\tq_{-2}+i\sum\limits_{j=1}^{2n} D_{\xi_j}d_1 D_{x_j}\tq_{-2}.
\end{split}
\label{7.2.8}
\end{equation}
Generally, we have
\begin{equation}
\begin{split}
q_{-(2+k)}(x,\xi)&=d_0(x)\tq_{-(2+k)}(x,\xi)+d_1(x,\xi)\tq_{-(3+k)}(x,\xi)+\\
&i\sum_{|\alpha|=1} D^{\alpha}_{\xi}
d_1(x,\xi)D^{\alpha}_{x}\tq_{-(2+k)}(x,\xi).
\end{split}
\label{7.6.5}
\end{equation}
Combining~\eqref{7.2.3} and~\eqref{7.2.6} shows that
\begin{equation}
q_{-2}=\sO_{-2}(|z|).
\label{7.6.4}
\end{equation}
Using the expressions in~\eqref{7.2.7} we see that for $k\geq 2$ we have
\begin{equation}
q_{-2k}=\sum_{j=0}^{l_k}\frac{\sO_{2j}(1)}{|\xi|^{2(k+j)}},\quad
q_{-(2k-1)}=\sum_{j=0}^{l_k'}\frac{\sO_{2j+1}(1)}{|\xi|^{2(k+j)}}.
\label{7.6.6}
\end{equation}

In order to compute the symbol of the Calderon projector, we  introduce
boundary adapted coordinates, $(t,x_2,\dots, x_{2n})$ where
\begin{equation}
t=-\frac{1}{2}\rho(z)=x_1+O(|x|^2).
\end{equation}
Note that $t$ is positive on a pseudoconvex manifold and negative on a
pseudoconcave manifold.

We need to use the change of coordinates formula to
express the symbol in the new variables. From~\cite{Hormander3} we obtain the
following prescription: Let $w=\phi(x)$ be a diffeomorphism and $a(x,\xi)$ the
symbol of a classical pseudodifferential operator $A$. Let $(w,\eta)$ be linear
coordinates in the cotangent space, then $a_\phi(w,\eta),$ the symbol of $A$
in the new coordinates, is given by
\begin{equation}
a_{\phi}(\phi(x),\eta)\sim
\sum_{k=0}^{\infty}\sum_{\alpha\in\cI_k}
\frac{(-i)^k\pa_\xi^\alpha a(x,d\phi(x)^t\eta)\pa_{\px}^\alpha e^{i\langle
    \Phi_x(\px),\eta\rangle}}
{\alpha!}\bigg|_{x=\px},
\label{7.6.1}
\end{equation}
where
\begin{equation}
\Phi_x(\px)=\phi(\px)-\phi(x)-d\phi(x)(\px-x).
\label{7.6.2}
\end{equation}
Here $\cI_k$ are multi-indices of length $k.$ Our symbols are matrix valued,
e.g. $q_{-2}$ is really $(q_{-2})_{pq}.$ As the change of variables applies
component by component, we suppress these indices in the computations that
follow.

In the case at hand, we are interested in evaluating this expression at $z=x=0,$
where we have $d\phi(0)=\Id$ and
$$\Phi_0(\px)=(-\frac{1}{2}[|\pz|^2+\Re(b\pz,\pz)+O(|\pz|^3)],\dots,0).$$ Note
also that, in~\eqref{7.6.1}, the symbol $a$ is only differentiated in the fiber
variables and, therefore, any term of the symbol that vanishes at $z=0,$ in the
K\"ahler coordinates, does not contribute to the symbol at $0$ in the boundary
adapted coordinates. Of particular importance is the fact that the term
$q_{-2}$ vanishes at $z=0$ and therefore does not contribute to the final
result. Indeed we shall see that only the principal symbol $q_{-1}$ contributes
to the Heisenberg principal symbol along the positive (or negative) contact
direction.

The $k=1$ term from~\eqref{7.6.1} vanishes, the $k=2$ term is given by
\begin{equation}
-\frac{i\xi_1}{2}\Tr[\pa^2_{\xi_j\xi_k} q(0,\xi)\pa^2_{x_jx_k}\phi(0)].
\end{equation}
For $k>2,$ the terms have the form
\begin{equation}
\sum_{\alpha\in\cI_k} \pa_{\xi}^{\alpha} q(0,\xi) p^{\alpha}(\xi_1).
\end{equation}
Here $p^{\alpha}$ is a polynomial of degree at most
$\lfloor\frac{|\alpha|}{2}\rfloor.$ As we shall see, the terms for $k>2$ do not
contribute to the final result.

To compute the $k=2$ term we need to compute the Hessians of $q_{-1}$ and
$\phi(x)$ at $x=0.$ 
We define the $2n\times 2n$ matrix $B$ so that
\begin{equation}
\Re(bz,z)=\langle B x,x\rangle;
\end{equation}
if $b=b^0+ib^1,$ then
\begin{equation}
B=\left(\begin{matrix} b^0 & -b^1\\
-b^1 & -b^0\end{matrix}\right).
\end{equation}
With these definitions we see that
\begin{equation}
\pa^2_{x_jx_k}\phi(0)=-(\Id+B).
\end{equation}

We further simplify the notation by letting $d_1(\xi)\overset{d}{=}d_1(0,\xi),$
then
\begin{equation}
q_{-1}(0,\xi)=\frac{2d_1(\xi)}{|\xi|^2}.
\end{equation}
Differentiating gives
\begin{equation}
\frac{\pa q_{-1}}{\pa\xi_j}=\frac{2\pa_{\xi_j} d_1}{|\xi|^2}-
\frac{2\xi_j d_1}{|\xi|^4}
\end{equation}
and
\begin{equation}
\frac{\pa^2q_{-1}}{\pa\xi_k\pa\xi_j}=
-4\frac{d_1\Id+\xi\otimes\pa_\xi d_1^t+\pa_\xi d_1\otimes\xi^t}{|\xi|^4}+
16d_1\frac{\xi\otimes\xi^t}{|\xi|^6}.
\end{equation}
Here $\xi$ and $\pa_\xi d_1$ are regarded as column vectors.

We now compute the principal part of the $k=2$ term
\begin{equation}
\begin{split}
q_{-2}^{c}(\xi)=&i\xi_1\Tr\left[(\Id+B)\left(-2\frac{d_1\Id+\xi\otimes\pa_\xi d_1^t+\pa_\xi
d_1\otimes\xi^t}{|\xi|^4}+ 8d_1\frac{\xi\otimes\xi^t}{|\xi|^6}\right)\right]\\
&= 4i\xi_1\left[(1-n)\frac{d_1}{|\xi|^4}+\frac{2d_1\langle
    B\xi,\xi\rangle}{|\xi|^6}-
\frac{\langle B\xi,\pa_\xi d_1\rangle}{|\xi|^4}\right].
\end{split}
\label{7.6.8}
\end{equation}

Because $q_{-2}$ vanishes at $0$ and because the order of a symbol is preserved
under a change of variables we 
see that  the symbol of $Q^{\eo}$ at $p$ is therefore
\begin{equation}
q(0,\xi)=\frac{2d_1(\xi)}{|\xi|^2}+q_{-2}^c(\xi)+\sO_{-3}(1).
\label{7.9.7}
\end{equation}
For the computation of the Calderon projector it is useful to be a little more
precise about the error term. The terms of highest symbolic order are multiples
of terms of the form $\xi^{k}_1\pa_{\xi}^{\alpha}q_{-j}$ where $|\alpha|=2k.$ We
describe, in a proposition, the types of terms that arise as error terms
in~\eqref{7.9.7}
\begin{proposition} The $\sO_{-3}(1)$-term in~\eqref{7.9.7} is a sum of terms
  of the form appearing in~\eqref{7.6.6} along with terms of the forms 
\begin{equation}
\begin{split}
&\frac{\xi_1^lh_{2m}(\xi)}{|\xi|^{2(k+l'+m)}} \text{ with either }
k=1\text{ and }l\geq 2\text{ or }k\geq 2\\
&\frac{\xi_1^lh_{2m+1}(\xi)}{|\xi|^{2(k+l'+m)}} \text{ with }
k\geq 2.
\end{split}
\label{7.9.8}
\end{equation}
Here $l'\geq l,$ $m$ is a nonnegative integer and $h_j(\xi)$ is a radially homogeneous
polynomial of degree $j.$ 
\end{proposition}
\begin{proof}
This statement is an immediate consequence of~\eqref{7.6.6},~\eqref{7.6.1} and
the fact that $\Phi_0(\px)$ vanishes quadratically at $\px=0.$
\end{proof}

\section{The symbol of the Calderon projector}\label{s.3}
We are now prepared to compute the symbol of the Calderon projector; it is
expressed as 1-variable contour integral in the symbol of $Q^{\eo}.$ If
$q(t,x',\xi_1,\xi')$ is the symbol of $Q^{\eo}$ in the boundary adapted
coordinates, then the symbol of the Calderon projector is
\begin{equation}
\scp_{\pm}(x',\xi')=\frac{1}{2\pi}\int\limits_{\Gamma_{\pm}(\xi_1)}q(0,x',\xi_1,\xi')d\xi_1
\circ\sigma_1(\eth^{\eo},\mp idt).
\label{7.9.1}
\end{equation}
Here we recall that $q(0,x',\xi_1,\xi')$ is a meromorphic function of $\xi_1.$
For each fixed $\xi',$ the poles of $q$ lie on the imaginary axis. If X is
strictly pseudoconvex, then $t>0$ on $X$ and we take $\Gamma_+(\xi_1)$ to be a
contour enclosing the poles of $q(0,x',\cdot,\xi')$ in the upper half plane. If
$X$ is strictly pseudoconcave, then $t<0$ on $X$ and $\Gamma_-(\xi_1)$ is a
contour enclosing the poles of $q(0,x',\cdot,\xi')$ in the lower half plane. In
a moment we  use a residue computation to evaluate these integrals. For
this purpose we note that the contour $\Gamma_+(\xi_1)$ is positively oriented,
while $\Gamma_-(\xi_1)$ is negatively oriented.

The Calderon projector is a classical pseudodifferential operator of order $0$ and
therefore its symbol has an asymptotic expansion of the form
\begin{equation}
\scp=\scp_{0}+\scp_{-1}+\dots
\end{equation}
The contact line, $L_p,$ is defined in $T_p^*Y$ by the equations
\begin{equation}
\xi_2=\dots=x_{n}=x_{n+2}=\dots=\xi_{2n}=0,
\end{equation}
and $\xi_{n+1}$ is a coordinate along the contact line. Because $t=-\ha\rho,$
the positive contact direction is given by $\xi_{n+1}<0.$  If $X$ is
pseudoconvex then, for $\xi'\notin L_p^+,$
it suffices to compute $\scp_0,$ whereas if $X$ is pseudoconcave, then for
$\xi'\notin L_p^-$ it suffices to compute $\scp_{0}.$ We begin our computations
with the principal symbol
\begin{proposition} If $X$ is strictly pseudoconvex (pseudoconcave) and  $p\in
  bX$ with coordinates normalized at $p$ as above, then
\begin{equation}
\scp^{\eo}_{0}(0,\xi')=
\frac{d_1^{\ooee}(\pm i|\xi'|,\xi')}{|\xi'|}\circ\sigma_1(\eth^{\eo},\mp i dt).
\label{7.9.3}
\end{equation}
\end{proposition}
\begin{proof} The leading term in the symbol of the Calderon projector comes from 
\begin{equation}
q_{-1}(0,\xi)=\frac{2
    d_1(\xi)}{|\xi|^2}=\frac{2d_1(\xi_1,\xi'))}{(\xi_1+i|\xi'|)(\xi_1-i|\xi'|)}.
\label{7.9.5}
\end{equation}
Evaluating the contour integral in~\eqref{7.9.1} gives~\eqref{7.9.3}.
\end{proof}

Along the contact directions we need to evaluate higher order terms.  We begin
by showing that the error terms in~\eqref{7.9.7} contribute terms that lift to
have Heisenberg order less than $-2$
\begin{proposition}\label{prop6} The error terms in~\eqref{7.9.7} contribute terms to the
  symbol of the Calderon projector that
  lift to have Heisenberg  orders at most $-4.$
\end{proposition}
\begin{proof} We first check the terms that come from the lower order terms in
  the symbol of $Q^{\eo}$ before changing variables. These are of the forms
  given in~\eqref{7.6.6} with $k\geq 2.$ It suffices to consider a term of the
  form
\begin{equation} 
\frac{h_{2j+1}(\xi)}{|\xi|^{2(k+j)}}
\end{equation}
for $k\geq 2$ and $j\geq 0.$ Applying the contour integration to such a term
gives a multiple of
\begin{equation}
\pa_{\xi_1}^{k+j-1}\left[\frac{h_{2j+1}(\xi)}{(\xi_1\pm
    i|\xi'|)^{k+j}}\right]_{\xi_1=\pm i|\xi'|}.
\end{equation}
As $\xi_{n+1}$ has Heisenberg order $2,$ it is not difficult to see that the
highest parabolic order term results if $h_{2j+1}(\xi)=\xi_{n+1}^{2j+1}.$
Differentiating gives a term of the form
\begin{equation}
\frac{\xi_{n+1}^{2j+1}}{|\xi'|^{2k+2j-1}}.
\end{equation}
Proposition~\ref{prop33} implies that this term lifts to have Heisenberg order
$4-4k.$ As $k\geq 2$ the proposition follows in this case.

Among the terms that come from the change of variables formula, there are two
cases to consider: those coming from $q_{-1}$ and those coming from $q_{-k}$
for $k\geq 3.$ Recall that $q_{-2}$ does not contribute anything to the symbol
at $p.$ The terms in~\eqref{7.6.1} coming from the principal symbol are of the
form
\begin{equation}
\frac{\xi_1^lh_{1+2j}(\xi)}{|\xi|^{2(1+j+l')}}\text{ where }2\leq l\leq
l'\text{ and }j\geq 0.
\end{equation}
Clearly the worst case is when $l=l'$ and $h_{2j+1}=\xi_{n+1}^{2j+1}.$ The
contour integral applied to such a term produces a multiple of
\begin{equation}
\frac{\xi_{n+1}^{2j+1}}{|\xi'|^{l+2j+1}}.
\end{equation}
This lifts to have Heisenberg order $-2l.$ As $l\geq 2,$ this completes the
analysis of the contribution of the principal symbol.

Finally we need to consider terms of the forms given in~\eqref{7.9.8} with
$k\geq 2$ and $l\geq 1.$ As before, the worse case is with $l=l'$ and
$h_{2j+1}(\xi)=\xi_{n+1}^{2j+1}.$ The contour integral gives a term of the form
\begin{equation}
\frac{\xi_{n+1}^{2j+1}}{|\xi'|^{2j+1}}\frac{1}{|\xi'|^{2k+l-2}}.
\end{equation}
As $2k+l\geq 5,$ these terms lift to have Heisenberg order at most $-6.$ This
completes the proof of the proposition.
\end{proof}

To finish our discussion of the symbol of the Calderon projector we need to
compute the symbol along the contact direction. This entails computing the
contribution from $q_{-2}^c.$ As we now show, terms arising from the
holomorphic Hessian of $\rho$ do not contribute anything to the symbol of the
Calderon projector. To do these computations we need to have an explicit
formula for the principal symbol $d_1(\xi)$ of $\eth$ at $p.$ For the purposes
of these and our subsequent computations, it is useful to use the chiral
operators $\eth^{\eo}.$ As we are working in a K\"ahler coordinate system, we
only need to find the symbols of $\eth^{\eo}$ for $\bbC^n$ with the flat
metric. Let $\sigma$ denote a section of $\Lambda^{\eo}\otimes E.$ We split $\sigma$
into its normal and tangential parts at $p:$
\begin{equation}
\sigma=\sigma^t+\frac{d\bz_1}{\sqrt{2}}\wedge\sigma^n,\quad
i_{\pa_{\bz_1}}\sigma^t=0, i_{\pa_{\bz_1}}\sigma^n=0.
\label{7.12.1}
\end{equation}

With this splitting we see that
\begin{equation}
\begin{split}
\eth^{\even}\sigma&=\sqrt{2}\left(\begin{matrix}
\pa_{\bz_1}\otimes\Id_{E,n} & \ccD_t\\
-\ccD_t & -\pa_{z_1}\otimes\Id_{E,n}\end{matrix}\right)\left(\begin{matrix} \sigma^t\\
\sigma^n\end{matrix}\right)\\
\eth^{\odd}\sigma&=\sqrt{2}\left(\begin{matrix}
-\pa_{z_1}\otimes\Id_{E,n} & -\ccD_t\\
\ccD_t & \pa_{\bz_1}\otimes\Id_{E,n}\end{matrix}\right)\left(\begin{matrix} \sigma^n\\
\sigma^t\end{matrix}\right),
\end{split}
\label{7.12.2}
\end{equation}
where $\Id_{E,n}$ is the identity matrix acting on the normal, or tangential
parts of $\Lambda^{\eo}\otimes E\restrictedto_{bX}$ and 
\begin{equation}
\ccD_t=\sum_{j=2}^n[\pa_{z_j} e_j-\pa_{\bz_j}\epsilon_j]\text{ with }
e_j= i_{\sqrt{2}\pa_{\bz_j}}\text{ and }
\epsilon_j=\frac{d\bz_j}{\sqrt{2}}\wedge.
\label{7.12.3}
\end{equation}
These symbols are expressed in the block matrix structure shown
in~\eqref{7.12.6}. It is now a simple matter to compute $d_1^{\eo}(\xi):$
\begin{equation}
\begin{split}
&d^{\even}_1(\xi)=\frac{1}{\sqrt{2}}\left(\begin{matrix}
(i\xi_1-\xi_{n+1})\otimes\Id_{E,n} & \sd(\xi'')\\
-\sd(\xi'') & -(i\xi_1+\xi_{n+1}) \otimes\Id_{E,n}\end{matrix}\right)\\
&d^{\odd}_1(\xi)=
\frac{1}{\sqrt{2}}\left(\begin{matrix}
-(i\xi_1+\xi_{n+1})\otimes\Id_{E,n}) & -\sd(\xi'')\\
\sd(\xi'') & (i\xi_1-\xi_{n+1})\otimes\Id_{E,n}\end{matrix}\right)
\end{split}
\label{7.12.4}
\end{equation}
where $\xi''=(\xi_2,\dots,\xi_n,\xi_{n+2},\dots,\xi_{2n})$ and
\begin{equation}
\sd(\xi'')=\sum_{j=2}^n[(i\xi_j+\xi_{n+j})e_j-(i\xi_j-\xi_{n+j})\epsilon_j].
\end{equation}
As $\epsilon_j^*=e_j$ we see that $\sd(\xi'')$ is a self adjoint symbol.

In the next section we show that, in the block structure shown in
equation~\eqref{7.12.6}, the $(1,1)$ block of the symbol of $\cT^{\eo}$ has
Heisenberg order $0,$ the $(1,2)$ and the $(2,1)$ blocks have Heisenberg order
$-1.$ The symbol $q_{-2}^c$ produces a term that lifts to have Heisenberg order
$-2$ and therefore we only need to compute the $(2,2)$ block arising from this
term.

We start with the nontrivial term of order $-1.$
\begin{lemma} If $X$ is either pseudoconvex or pseudoconcave we have that
\begin{equation}
\frac{1}{2\pi}\int\limits_{\Gamma_{\pm}(\xi')}\frac{4i\xi_1(1-n)
  d_1(\xi_1,\xi')d\xi_1}{|\xi|^4}=-\frac{i(n-1)\pa_{\xi_1} d_1}{|\xi'|}
\label{7.12.7}
\end{equation}
\end{lemma}
\begin{remark} As $d_1$ is a linear polynomial, $\pa_{\xi_1}d_1$ is a constant
  matrix. 
\end{remark}
\begin{proof}
The residue theorem implies that
\begin{equation}
\frac{1}{2\pi}\int\limits_{\Gamma_{\pm}(\xi')}\frac{4i\xi_1(1-n)
  d_1(\xi_1,\xi')d\xi_1}{|\xi|^4}=
\pm 4(n-1)\pa_{\xi_1}\left[\frac{\xi_1 d_1}{(\xi_1\pm i|\xi'|)^2}\right]_{\xi_1=\pm i|\xi'|.}
\end{equation}
The lemma follows from this equation by an elementary computation.
\end{proof} 

We complete the computation by evaluating the contribution from the other terms
in $q_{-2}^c$ along the contact line.
\begin{proposition}
For $\xi'$ along
the positive (negative) contact line we have, for $j=1,2,$ that
\begin{equation}
\int\limits_{\Gamma_{\pm}(\xi')}\left[\frac{2d_1(\xi)\langle B\xi,\xi\rangle-|\xi|^2
\langle B\xi,\pa_\xi d_1\rangle}{|\xi|^6}\right]_{jj}\xi_1d\xi_1=0.
\label{7.12.8}
\end{equation}
The subscript $11$ refers to the upper left block and $22$ the lower right
block of the matrix.  If $\xi_{n+1}<0,$ then we use $\Gamma_+(\xi'),$ whereas
if $\xi_{n+1}>0,$ then we use $\Gamma_-(\xi').$
\end{proposition}
\begin{proof}
To prove this result we need to evaluate the contour integral with
$$\xi'=\xi'_c=(0,\dots,0,\xi_{n+1},0,\dots,0),$$
recalling that the positive contact line corresponds to $\xi_{n+1}<0.$ Hence,
along the positive contact line $|\xi'|=-\xi_{n+1}.$  Because 
\begin{equation}
[d_1^{\even}]_{11}=[d_1^{\odd}]_{22}\text{ and
}[d_1^{\even}]_{22}=[d_1^{\odd}]_{11},
\end{equation}
it suffices to prove the result for the $(2,2)$ block in both the even and odd
cases. We first compute the integrand along $\xi'_c.$ 
\begin{lemma} For $\xi'$ along the contact line we have
\begin{multline}
\left[\frac{2d_1^{\even}(\xi)\langle B\xi,\xi\rangle-|\xi|^2
\langle B\xi,\pa_\xi d_1^{\even}\rangle}{|\xi|^6}\right]_{22}=\\
\frac{(\xi_1 b^1_{11}+\xi_{n+1} b^0_{11})+i(\xi_1b^0_{11}-\xi_{n+1}b^1_{11})}
{(\xi_{n+1}+i\xi_1)(\xi_{n+1}-i\xi_1)^3}\otimes\Id_{E,n}
\label{7.12.9}
\end{multline}
\begin{multline}
\left[\frac{2d_1^{\odd}(\xi)\langle B\xi,\xi\rangle-|\xi|^2
\langle B\xi,\pa_\xi d_1^{\odd}\rangle}{|\xi|^6}\right]_{22}=\\
\frac{ (\xi_1b^1_{11}+\xi_{n+1} b^0_{11})-i(\xi_1 b^0_{11}-\xi_{n+1} b^1_{11})}
{(\xi_{n+1}-i\xi_1)(\xi_{n+1}+i\xi_1)^3}\otimes\Id_{E,n}.
\label{7.12.10}
\end{multline}
The subscript $22$ refers to the lower right block of the matrix.
\end{lemma} 
\begin{proof}
Observe that along the contact line
$$\langle B\xi,\xi\rangle=b^0_{11}(\xi_1^2-\xi_{n+1}^2)-2b^1_{11}\xi_1\xi_{n+1}.$$
We outline the proof for the even case. The lower right block
of $d_1^{\even}(\xi)$ equals $-(i\xi_1+\xi_{n+1})\otimes\Id_{E,n},$ thus
$$\left[\pa_{\xi}d_1^{\even}\right]_{22}=
(-i,0\dots,0,-1,0,\dots,0)\otimes\Id_{E,n}.$$
Putting these expressions into the formula on the left hand side of~\eqref{7.12.9} gives
$\Id_{E,n}$ times
\begin{multline}
\frac{-2(i\xi_1+\xi_{n+1})(b^0_{11}(\xi_1^2-\xi_{n+1}^2)-2b^1_{11}\xi_1\xi_{n+1})}{|\xi|^3}\\
-\frac{(\xi_1b^1_{11}+\xi_{n+1}b^0_{11})-i(\xi_1b^0_{11}-\xi_{n+1}b^1_{11})}
{|\xi|^4}.
\end{multline}
To complete the calculation we express
$|\xi|^2=(\xi_{n+1}+i\xi_1)(\xi_{n+1}-i\xi_1),$ cancel and place the result over
a common denominator. This leads to the cancellation of a second factor of
$\xi_{n+1}+i\xi_1.$ The odd case follows, \emph{mutatis mutandis}, using
$$d_1^{\odd}(\xi)=(i\xi_1-\xi_{n+1})\otimes\Id_{E,n}.$$
The details are left to the reader.
\end{proof}

To complete the proof of the proposition we need to compute the contour
integrals of the expressions in~\eqref{7.12.9} and~\eqref{7.12.10} times
$\xi_1,$ along the appropriate end of the contact line. We state these
computations as lemmas.
\begin{lemma} If $\xi_{n+1}<0,$ then
\begin{equation}
\begin{split}
\text{even}\quad&\int\limits_{\Gamma_+(\xi'_c)}\frac{(\xi_1 b^1_{11}+\xi_{n+1} b^0_{11})+
i(\xi_1b^0_{11}-\xi_{n+1}b^1_{11})}
{(\xi_{1}-i\xi_{n+1})(\xi_{1}+i\xi_{n+1})^3}\xi_1 d\xi_1=0\\
\text{odd}\quad&\int\limits_{\Gamma_+(\xi'_c)}\frac{ (\xi_1b^1_{11}+\xi_{n+1} b^0_{11})-
i(\xi_1 b^0_{11}-\xi_{n+1} b^1_{11})}
{(\xi_{1}+i\xi_{n+1})(\xi_{1}-i\xi_{n+1})^3}\xi_1d\xi_1=0
\end{split}
\end{equation}
Note that this implies that, if $\xi_{n+1}>0,$ then the same integrals vanish if
$\Gamma_+(\xi'_c)$ is replaced by $\Gamma_-(\xi'_c).$
\end{lemma}
\begin{proof} The second statement follows by observing that the singular terms 
  in the integrand in the upper half plane are those coming from
  $(\xi_1+i\xi_{n+1}).$ If $\xi_{n+1}>0,$ then these become the singular terms
  in the lower half plane. Using a residue computation we see that the even
  case gives
\begin{equation}
\begin{split}
&(\pi i)\pa_{\xi_1}^2\left[\frac{(\xi_1 b^1_{11}+\xi_{n+1} b^0_{11})+
i(\xi_1b^0_{11}-\xi_{n+1}b^1_{11})}
{(\xi_{1}-i\xi_{n+1})}\right]_{\xi_1=-i\xi_{n+1}}\\
&=\frac{2\pi}{(-2i\xi_{n+1})^2}\left[
b^1_{11}+ib^0_{11}-\frac{(\xi_1 b^1_{11}+\xi_{n+1}b^0_{11})+i(\xi_1
  b^0_{11}-\xi_{n+1} b^1_{11})}{\xi_1-i\xi_{n+1}}\right]_{\xi_1=-i\xi_{n+1}}.
\end{split}
\end{equation}
The quantity in the brackets is easily seen to vanish. The odd case follows
easily from the observation that
\begin{equation}
\left[(\xi_1b^1_{11}+\xi_{n+1} b^0_{11})-
i(\xi_1 b^0_{11}-\xi_{n+1} b^1_{11})\right]_{\xi_1=-i\xi_{n+1}}=0.
\end{equation}
\end{proof}
The two lemmas prove the proposition.
\end{proof}

As a corollary, we have a formula for the $-1$ order term in the symbol of the
Calderon projector
\begin{corollary}
If $X$ is strictly pseudoconvex (pseudoconcave), then, in the normalizations
defined above, for $j=1,2,$  we have
\begin{equation}
[\scp_{-1}^{\eo}(0,\xi')]_{jj}
=-\frac{i(n-1)\pa_{\xi_1}d_1^{\ooee}}{|\xi'|}\circ\sigma_1(\eth^{\eo},\mp
idt).
\end{equation}
\end{corollary}

We have shown that the order $-1$ term in the symbol of the Calderon projector,
along the appropriate half of the contact line, is given by the right hand side of
equation~\eqref{7.12.7}. It is determined by the principal symbol of $Q^{\eo}$
and does not depend on the higher order geometry of $bX.$ As we have shown that
all other terms in the symbol of $Q^{\eo}$ contribute terms that lift to have
Heisenberg order less than $-2,$ these computations allow us to find the
principal symbols of $\cT^{\eo}_{\pm}$ and deduce the main results of the
paper. As noted above, the off diagonal blocks have Heisenberg order $-1,$ so
the classical terms of order less than zero cannot contribute to their
principal parts. 

\section{The subelliptic boundary conditions}
We now give formul\ae\ for the chiral forms of the subelliptic boundary
conditions defined in~\cite{Epstein4} as well as the isomorphisms
$\sigma_1(\eth^{\eo},\mp i dt).$ We begin by recalling the basic properties of
compatible almost complex structures defined on a contact field and of the
symbol of a generalized Szeg\H o projector. Let $\theta$ denote a positive
contact form defining $H.$ An almost complex structure on $H$ is compatible if
\begin{enumerate}
\item $X\mapsto d\theta(JX,X)$ defines an inner product on $H.$
\item $d\theta(JX,JY)=d\theta(X,Y)$ for sections of $H.$
\end{enumerate}
Let $\omega'$ be the dual symplectic form on $H^*$ and $J'$ the dual almost
complex structure. The symbol of a field of harmonic oscillators is defined by
\begin{equation}
h_J(\eta)=\omega'(J'\pi_{H^*}(\eta),\pi_{H^*}(\eta)).
\end{equation}
The model operator defined by the symbol $h_J$ is a harmonic oscillator, as
such its minimum eigenstate or vacuum state is one dimensional. The projector
onto the vacuum state has symbol $s_{J0}=2^{1-n}e^{-h_J}.$ An operator $\cS'$ in the Heisenberg
calculus with principal symbol $s_{J0},$ for a compatible almost complex
structure $J,$ such that
\begin{equation}
[\cS']^2=\cS'\text{ and }[\cS']^*=\cS'
\label{7.26.4}
\end{equation}
is called a generalized Szeg\H o projector. Generalized conjugate Szeg\H o
projectors are analogously defined, with the symbol supported on the lower half
space. A generalized Szeg\H o projector acting on
sections of a complex vector bundle $F\to bX$ is an operator in $\Psi^0_H(Y;F),$
which satisfies the conditions
in~\eqref{7.26.4} and its principal symbol is $s_{J0}\otimes\Id_F.$

\begin{lemma}\label{lemm4} According to the splittings of sections of $\Lambda^{\eo}\otimes E$
  given in~\eqref{7.12.1}, the subelliptic boundary conditions, defined by the
  generalized Szeg\H o projector $\cS',$  on even (odd)
  forms are given by $\cR^{\prime\eo}_+\sigma\restrictedto_{bX}=0$ where
\begin{equation}
\cR^{\prime\even}_+\sigma\restrictedto_{bX}=
\left(\begin{matrix}\begin{matrix} \cS' & 0 \\
0 & \bzero\\
\end{matrix} &
\begin{matrix} \bzero
\end{matrix}\\
\begin{matrix}\bzero
\end{matrix}&  
\begin{matrix}\Id \end{matrix} \end{matrix}\right)
\left[\begin{matrix} \sigma^{t}\\\phantom{\sigma}\\
\sigma^{n}
\end{matrix}\right]_{bX}
\quad
\cR^{\prime\odd}_+\sigma\restrictedto_{bX}=
\left(\begin{matrix}\begin{matrix} 1-\cS' & 0 \\
0 & \Id\\
\end{matrix} &
\begin{matrix} \bzero
\end{matrix}\\
\begin{matrix}\bzero
\end{matrix}&  
\begin{matrix}\bzero \end{matrix} \end{matrix}\right)
\left[\begin{matrix} \sigma^{n}\\\phantom{\sigma}\\
\sigma^{t}
\end{matrix}\right]_{bX}
\end{equation}
\end{lemma}

\begin{lemma}\label{lemm5} According to the splittings of sections of $\Lambda^{\eo}\otimes E$
  given in~\eqref{7.12.1}, the subelliptic boundary conditions, defined by the
  generalized conjugate Szeg\H o projector $\bcS',$  on even (odd)
  forms are given by $\cR^{\prime\eo}_-\sigma\restrictedto_{bX}=0$ where, if $n$ is
  even, then
\begin{equation}
\cR^{\prime\even}_-\sigma\restrictedto_{bX}=
\left(\begin{matrix}\begin{matrix}\bzero\end{matrix}&
\begin{matrix} \bzero
\end{matrix}\\
\begin{matrix}\bzero
\end{matrix}&  \begin{matrix} \Id & 0 \\
0 & 1-\bcS'
\end{matrix} 
\end{matrix}\right)
\left[\begin{matrix} \sigma^{t}\\\phantom{\sigma}\\
\sigma^{n}
\end{matrix}\right]_{bX}
\quad
\cR^{\prime\odd}_-\sigma\restrictedto_{bX}=
\left(\begin{matrix}\begin{matrix} \Id 
\end{matrix} &
\begin{matrix} \bzero
\end{matrix}\\
\begin{matrix}\bzero
\end{matrix}&  
\begin{matrix}\bzero & 0\\ 0 &\bcS' \end{matrix} \end{matrix}\right)
\left[\begin{matrix} \sigma^{n}\\\phantom{\sigma}\\
\sigma^{t}
\end{matrix}\right]_{bX}.
\end{equation}
If $n$ is odd, then
\begin{equation}
\cR^{\prime\even}_-\sigma\restrictedto_{bX}=
\left(\begin{matrix}\begin{matrix}\bzero& 0\\ 0&\bcS'\end{matrix}&
\begin{matrix} \bzero
\end{matrix}\\
\begin{matrix}\bzero
\end{matrix}&  \begin{matrix} \Id
\end{matrix} 
\end{matrix}\right)
\left[\begin{matrix} \sigma^{t}\\\phantom{\sigma}\\
\sigma^{n}
\end{matrix}\right]_{bX}
\quad
\cR^{\prime\odd}_-\sigma\restrictedto_{bX}=
\left(\begin{matrix}\begin{matrix} \Id & 0\\
0 &1-\bcS'
\end{matrix} &
\begin{matrix} \bzero
\end{matrix}\\
\begin{matrix}\bzero
\end{matrix}&  
\begin{matrix}\bzero\end{matrix} \end{matrix}\right)
\left[\begin{matrix} \sigma^{n}\\\phantom{\sigma}\\
\sigma^{t}
\end{matrix}\right]_{bX}.
\end{equation}
\end{lemma}
\begin{remark} These boundary conditions are introduced in~\cite{Epstein4}. For the
  purposes of this paper, these formul\ae\ can be taken as the definitions of
  the projections $\cR^{\prime\eo}_{\pm},$ which, in turn, define the boundary conditions.
\end{remark}

\begin{lemma}
The isomorphisms at the boundary between $\Lambda^{\eo}\otimes E$ and 
$\Lambda^{\ooee}\otimes E$ are given by
\begin{equation}
\sigma_1(\eth^{\eo}_{\pm},\mp i dt)\sigma^t=\frac{\pm }{\sqrt{2}}\sigma^t,\quad
\sigma_1(\eth^{\eo}_{\pm},\mp i dt)\sigma^n=\frac{\mp}{\sqrt{2}}\sigma^n.
\label{7.13.3}
\end{equation}
\end{lemma}

We have thus far succeeded in computing the symbols of the Calderon projectors
to high enough order to compute the principal symbols of $\cT^{\eo}_{\pm}$ as
elements of the extended Heisenberg calculus. The computations have been carried
out in a coordinate system adapted to the boundary. This suffices to examine
the classical parts of the symbols. In the next section we further normalize
the coordinates, in order to analyze the Heisenberg symbols.

We close this section by computing the classical parts of the symbols of
$\cT^{\eo}_{\pm}$ and showing that they are invertible on the complement of the
appropriate half of the contact line. Recall that the positive contact ray,
$L^+,$  is given at $p$ by $\xi''=0,\xi_{n+1}<0.$
\begin{proposition}\label{prp8} If $X$ is strictly pseudoconvex, then, on the complement of
  the positive contact direction, the classical symbols
  ${}^R\sigma_0(\cT^{\eo}_+)$ are given by
\begin{equation}
\begin{split}
{}^R\sigma_0(\cT^{\even}_+)(0,\xi')&=
\frac{1}{2|\xi'|}\left(\begin{matrix} (|\xi'|+\xi_{n+1})\Id & -\sd(\xi'')\\
\sd(\xi'') & (|\xi'|+\xi_{n+1})\Id \end{matrix}\right)\\
{}^R\sigma_0(\cT^{\odd}_+)(0,\xi')&=
\frac{1}{2|\xi'|}\left(\begin{matrix} (|\xi'|+\xi_{n+1})\Id & \sd(\xi'')\\
-\sd(\xi'') & (|\xi'|+\xi_{n+1})\Id \end{matrix}\right)
\end{split}
\label{7.13.1}
\end{equation}
These symbols are invertible on the complement of $L^+.$
\end{proposition}
 \begin{proof} Away from the positive contact direction $\cR^{\prime\eo}_+$ are
   classical pseudodifferential operators with
\begin{equation}
{}^R\sigma_0(\cR^{\prime\even}_+)=\left(\begin{matrix} \bzero &\bzero \\
\bzero &\Id\end{matrix}\right),\quad
{}^R\sigma_0(\cR^{\prime\odd}_+)=\left(\begin{matrix} \Id &\bzero \\
\bzero &\bzero\end{matrix}\right)
\label{7.13.2}
\end{equation}
The formul\ae\ in~\eqref{7.13.1} follow easily from these
relations, along with~\eqref{7.9.3},~\eqref{7.12.4}, and~\eqref{7.13.3}. To
show that these
symbols are invertible away from the positive contact direction, it suffices to
show that their determinants do not vanish. Up to the factor of
$(2|\xi'|)^{-1},$ these symbols are of the form $\lambda\Id+B$ where $\lambda$
is real (and nonnegative) and $B$ is skew-adjoint. As a skew-adjoint matrix has
purely imaginary spectrum, the determinants of these symbols vanish if and
only if $\sd(\xi'')=0$ and $|\xi'|+\xi_{n+1}=0.$ The first condition implies
that $|\xi'|=|\xi_{n+1}|,$ hence these determinant vanish if and only if $\xi'$
belongs to the positive contact ray.
\end{proof}

An essentially identical argument, taking into account the fact that
$\cR^{\prime\eo}_-$ are classical pseudodifferential operators on the
complement of $L^-,$  suffices to treat the pseudoconcave case.
\begin{proposition}\label{prp9} If $X$ is strictly pseudoconcave, then, on the complement of
  the negative contact direction, the classical symbols
  ${}^R\sigma_0(\cT^{\eo}_-)$ are given by
\begin{equation}
\begin{split}
{}^R\sigma_0(\cT^{\even}_-)(0,\xi')&=
\frac{1}{2|\xi'|}\left(\begin{matrix} (|\xi'|-\xi_{n+1})\Id & \sd(\xi'')\\
-\sd(\xi'') & (|\xi'|-\xi_{n+1})\Id \end{matrix}\right)\\
{}^R\sigma_0(\cT^{\odd}_-)(0,\xi')&=
\frac{1}{2|\xi'|}\left(\begin{matrix} (|\xi'|-\xi_{n+1})\Id & -\sd(\xi'')\\
\sd(\xi'') & (|\xi'|-\xi_{n+1})\Id \end{matrix}\right)
\end{split}
\label{7.13.4}
\end{equation}
These symbols are invertible  on the complement of $L^-.$
\end{proposition}

\begin{remark} Propositions~\ref{prp8} and~\ref{prp9}  are classical and
  implicitly stated, for example, in the work of Greiner and Stein, and Beals
  and Stanton, see~\cite{Beals-Stanton,Greiner-Stein1}.
\end{remark}

\section{The Heisenberg symbols of $\cT^{\eo}_{\pm}$}
To compute the Heisenberg symbols of $\cT^{\eo}_{\pm}$ we change coordinates,
one last time, to get Darboux coordinates at $p.$ Up to this point we have used
the coordinates $(\xi_2,\dots,\xi_{2n})$ for $T^*_pbX,$ which are defined by
the coframe $dx_2,\dots,dx_{2n},$ with $dx_{n+1}$ the contact direction. Recall
that the contact form $\theta,$ defined by the complex structure and defining
function $\rho/2,$ is given by $\theta=\frac{i}{2}\dbar\rho.$ The symplectic
form on $H$ is defined by $d\theta.$ At $p$ we have
\begin{equation}
\theta_p=-\frac{1}{2}dx_{n+1},\quad d\theta_p=\sum_{j=2}^{n}dx_j\wedge
dx_{j+n}.
\label{7.14.1}
\end{equation}
By comparison with~\eqref{7.14.3}, we see that properly normalized coordinates
for $T^*_pbX$ are obtained by setting
\begin{equation}
\eta_0=-2\xi_{n+1},\quad \eta_j=\xi_{j+1},\quad
\eta_{j+n-1}=\xi_{j+n+1}\text{ for }j=1,\dots,n-1.
\end{equation}
As usual we let $\eta'=(\eta_1,\dots,\eta_{2(n-1)});$ whence $\xi''=\eta'.$

As a first step in lifting the symbols of the Calderon projectors to the
extended Heisenberg compactification, we re-express them, through order $-1$ in
the $\xi$-coordinates:
\begin{equation}
\scp^{\even}_+(\xi')=
\frac{1}{2|\xi'|}\left[
\left(\begin{matrix} (|\xi'|-\xi_{n+1})\Id & \sd(\xi'')\\\sd(\xi'')&
(|\xi'|+\xi_{n+1})\Id\end{matrix}\right)-(n-1)\left(\begin{matrix}\Id &\bzero\\
\bzero&\Id\end{matrix}\right)\right]
\label{7.14.5}
\end{equation}

\begin{equation}
\scp^{\odd}_+(\xi')=
\frac{1}{2|\xi'|}\left[
\left(\begin{matrix} (|\xi'|+\xi_{n+1})\Id & \sd(\xi'')\\\sd(\xi'')&
(|\xi'|-\xi_{n+1})\Id\end{matrix}\right)-(n-1)\left(\begin{matrix}\Id &\bzero\\
\bzero&\Id\end{matrix}\right)\right]
\label{7.14.6}
\end{equation}

\begin{equation}
\scp^{\even}_-(\xi')=
\frac{1}{2|\xi'|}\left[
\left(\begin{matrix} (|\xi'|+\xi_{n+1})\Id & -\sd(\xi'')\\-\sd(\xi'')&
(|\xi'|-\xi_{n+1})\Id\end{matrix}\right)+(n-1)\left(\begin{matrix}\Id &\bzero\\
\bzero&\Id\end{matrix}\right)\right]
\label{7.14.7}
\end{equation}

\begin{equation}
\scp^{\odd}_-(\xi')=
\frac{1}{2|\xi'|}\left[
\left(\begin{matrix} (|\xi'|-\xi_{n+1})\Id & -\sd(\xi'')\\-\sd(\xi'')&
(|\xi'|+\xi_{n+1})\Id\end{matrix}\right)+(n-1)\left(\begin{matrix}\Id &\bzero\\
\bzero&\Id\end{matrix}\right)\right]
\label{7.14.8}
\end{equation}
Various identity and zero matrices appear in these symbolic
computations. Precisely which matrix is needed depends on the dimension, the
bundle $E,$ the parity, etc. We do not encumber our notation with these
distinctions.

In order to compute ${}^H\sigma(\cT^{\eo}_{\pm}),$ we represent the Heisenberg
symbols as model operators and use operator composition.  To that end we need
to quantize $\sd(\eta')$ as well as the terms coming from the diagonals
in~\eqref{7.14.5}--~\eqref{7.14.8}.  We first treat the pseudoconvex side. In
this case we need to consider the symbols on positive Heisenberg face, where
the function $|\xi'|+\xi_{n+1}$ vanishes. 

We express the various terms in $\scp^{\eo}_{+},$ near the positive contact line
as sums of Heisenberg homogeneous terms
\begin{equation}
\begin{split}
|\xi'|=&\frac{\eta_0}{2}(1+\sO_{-2}^H)\\
 |\xi'|-\xi_{n+1}=\eta_0(1+\sO_{-2}^H),\quad&
|\xi'|+\xi_{n+1}=\frac{|\eta'|^2}{\eta_0}(1+\sO_{-2}^H)\\
\sd(\xi'')=\sum_{j=1}^{n-1}[(i\eta_j+\eta_{n+j-1})e_j&-(i\eta_j-\eta_{n+j-1})\epsilon_j].
\end{split}
\label{7.14.9}
\end{equation}
Recall that the notation $\sO_j^H$ denotes a term of Heisenberg order at most
$j.$ To find the model operators, we split $\eta'=(w,\varphi).$ Using the
quantization rule in~\eqref{7.7.3} (with the $+$ sign) we see that
\begin{equation}
\begin{split}
\eta_j-i\eta_{n+j-1}&\leftrightarrow C_j\overset{d}{=}(w_j-\pa_{w_j})\\
\eta_j+i\eta_{n+j-1}&\leftrightarrow C_j^*\overset{d}{=}(w_j+\pa_{w_j})\\
|\eta'|^2&\leftrightarrow \ho\overset{d}{=}\sum_{j=1}^{n-1}w_j^2-\pa_{w_j}^2.
\end{split}
\label{7.14.10}
\end{equation}
The following standard identities are  useful
\begin{equation}
\sum_{j=1}^{n-1}C_j^*C_j-(n-1)=\ho=\sum_{j=1}^{n-1}C_jC_j^*+(n-1)
\label{7.14.11}
\end{equation}
We let $\cD_+$ denote the model operator defined, using the $+$ quantization,
by $\sd(\xi''),$ it is given by
\begin{equation}
\cD_+=i\sum_{j=1}^{n-1}[C_j e_j-C_j^*\epsilon_j].
\end{equation}
This is the model operator defined by $\dbarb+\dbarb^*$ acting on
$\oplus_q\Lambda_b^{0,q}\otimes E.$ This operator can be split into even and odd parts,
$\cD^{\eo}_+$ and these chiral forms of the operator are what appear in the
model operators below. To keep the notation from becoming too complicated we
suppress this dependence.

With these preliminaries, we can compute the model operators for $\cP^{\even}_+$
and $\Id-\cP^{\even}_+$ in the positive contact direction. They are:
\begin{equation}
{}^{eH}\sigma(\cP^{\even}_+)(+)=\left(\begin{matrix} \Id &\frac{\cD_+}{\eta_0}\\\frac{\cD_+}{\eta_0}
  &\frac{\ho-(n-1)}{\eta_0^2}\end{matrix}\right)\quad
{}^{eH}\sigma(\Id-\cP^{\even}_+)(+)=\left(\begin{matrix} \frac{\ho+n-1}{\eta_0^2} 
&-\frac{\cD_+}{\eta_0}\\-\frac{\cD_+}{\eta_0} &\Id\end{matrix}\right).
\label{7.14.12}
\end{equation}
The denominators involving $\eta_0$ are meant to remind the reader of the
Heisenberg orders of the various blocks: $\eta_0^{-1}$ indicates a term of
Heisenberg order $-1$ and $\eta_0^{-2}$ a term of order $-2.$ Similar
computations give the model operators in the odd case:
\begin{equation}
{}^{eH}\sigma(\cP^{\odd}_+)(+)=\left(\begin{matrix} \frac{\ho-(n-1)}{\eta_0^2}  &\frac{\cD_+}{\eta_0}\\\frac{\cD_+}{\eta_0}
  &\Id\end{matrix}\right)\quad
{}^{eH}\sigma(\Id-\cP^{\odd}_+)(+)=\left(\begin{matrix} \Id
&-\frac{\cD_+}{\eta_0}\\-\frac{\cD_+}{\eta_0} &\frac{\ho+n-1}{\eta_0^2} \end{matrix}\right).
\label{7.14.13}
\end{equation}
Let $\sypr_0'={}^{eH}\sigma(+)(\cS');$ this is a self adjoint rank one
projection defined by a compatible almost complex structure on $H,$ then
\begin{equation}
{}^{eH}\sigma(\cR^{\prime\even}_{+})(+)=
\left(\begin{matrix}\begin{matrix} \sypr'_0 & 0 \\
0 & \bzero\\
\end{matrix} &
\begin{matrix} \bzero
\end{matrix}\\
\begin{matrix}\bzero
\end{matrix}&  
\begin{matrix}\Id \end{matrix} \end{matrix}\right),\quad
{}^{eH}\sigma(\cR^{\prime\odd})(+)=
\left(\begin{matrix}\begin{matrix} 1-\sypr'_0 & 0 \\
0 & \Id
\end{matrix} &
\begin{matrix} \bzero
\end{matrix}\\
\begin{matrix}\bzero
\end{matrix}&  
\begin{matrix}\bzero \end{matrix} \end{matrix}\right).
\end{equation}
We can now compute the model operators for $\cT^{\eo}_+$
on the upper Heisenberg face.
\begin{proposition} If $X$ is strictly pseudoconvex, then, at $p\in bX,$ the
  model operators for $\cT^{\eo}_+,$ in the positive contact direction, are given by
\begin{equation}
{}^{eH}\sigma(\cT^{\even}_+)(+)=
\left(\begin{matrix}\begin{matrix} \sypr'_0 & 0 \\
0 & \bzero\\
\end{matrix} &\begin{matrix}-\left[\begin{matrix} 1-2\sypr'_0 & 0 \\
0 &\Id
\end{matrix}\right]\frac{\cD_+}{\eta_0}
\end{matrix}
\\
\begin{matrix} \frac{\cD_+}{\eta_0}
\end{matrix}&  
\begin{matrix} \frac{\ho-(n-1)}{\eta_0^2} \end{matrix} \end{matrix}\right)
\label{7.14.21}
\end{equation}
\begin{equation}
{}^{eH}\sigma(\cT^{\odd}_+)(+)=
\left(\begin{matrix}\begin{matrix} \sypr'_0 & 0 \\
0 & \bzero\\
\end{matrix} &\begin{matrix}\left[\begin{matrix} 1-2\sypr'_0 & 0 \\
0 &\Id
\end{matrix}\right]\frac{\cD_+}{\eta_0}
\end{matrix}
\\
\begin{matrix} -\frac{\cD_+}{\eta_0}
\end{matrix}&  
\begin{matrix} \frac{\ho+(n-1)}{\eta_0^2} \end{matrix} \end{matrix}\right).
\label{7.14.22}
\end{equation}
\end{proposition}
\begin{proof} Observe that the Heisenberg orders of the blocks
  in~\eqref{7.14.21} and~\eqref{7.14.22} are
\begin{equation}
\left(\begin{matrix} 0 & -1 \\-1 & -2\end{matrix}\right).
\end{equation}
Proposition~\ref{prop6} shows that all other terms in the symbol of the
Calderon projector lead to diagonal terms of Heisenberg order at most $-4,$ and
off diagonal terms of order at most $-2.$ This, along
with the computations above, completes the proof of the proposition.
\end{proof}

A similar analysis applies for the pseudoconcave case. Here we use that, near
the negative contact line, we have
\begin{equation}
\begin{split}
|\xi'|=&-\frac{\eta_0}{2}(1+\sO_{-2}^{H})\\
 |\xi'|+\xi_{n+1}=-\eta_0(1+\sO_{-2}^H),\quad&
|\xi'|-\xi_{n+1}=-\frac{|\eta'|^2}{\eta_0}(1+\sO_{-2}^H)
\end{split}
\label{7.14.99}
\end{equation}
The formula for $\sd(\xi'')$ is the same, however, the quantization rule is
slightly different, note the $\pm$ in equation~\eqref{7.7.3}. Using the $-$
sign we get the following quantizations:
\begin{equation}
\begin{split}
\eta_j-i\eta_{n+j-1}&\leftrightarrow C_j^*\overset{d}{=}(w_j+\pa_{w_j})\\
\eta_j+i\eta_{n+j-1}&\leftrightarrow C_j\overset{d}{=}(w_j-\pa_{w_j})\\
|\eta'|^2&\leftrightarrow \ho\overset{d}{=}\sum_{j=1}^{n-1}w_j^2-\pa_{w_j}^2.
\end{split}
\label{7.14.20}
\end{equation}
With the $-$ sign we therefore obtain that the model operator defined by
$\sd(\xi'')$ is
\begin{equation}
\cD_-=i\sum_{j=1}^{n-1}[C_j^* e_j-C_j\epsilon_j].
\end{equation}
Using computations identical to those above, we find
that the model operators for $\cP^{\eo}_-$ and $\Id-\cP^{\eo}_{-},$ along the
negative contact direction are:
\begin{equation}
\begin{split}
&\esym{eH}{}(\cP^{\even}_-)(-)=\left(\begin{matrix} \Id & \frac{\cD_-}{|\eta_0|}\\
\frac{\cD_-}{|\eta_0|} & \frac{\ho+ (n-1)}{|\eta_0|^2}\end{matrix}\right)\quad
\esym{eH}{}(\Id-\cP^{\even}_-)(-)=\left(\begin{matrix}
  \frac{\ho-(n-1)}{|\eta_0|^2} & 
-\frac{\cD_-}{|\eta_0|}\\ -\frac{\cD_-}{|\eta_0|} &\Id \end{matrix}\right)\\
&\esym{eH}{}(\cP^{\odd}_-)(-)=\left(\begin{matrix} \frac{\ho+(n-1)}{|\eta_0|^2}
& \frac{\cD_-}{|\eta_0|}\\
\frac{\cD_-}{|\eta_0|} & \Id\end{matrix}\right)\quad
\esym{eH}{}(\Id-\cP^{\odd}_-)(-)=\left(\begin{matrix}  \Id & -\frac{\cD_-}{|\eta_0|}\\ 
-\frac{\cD_-}{|\eta_0|} &\frac{\ho-(n-1)}{|\eta_0|^2}\end{matrix}\right).
\end{split}
\label{7.15.1}
\end{equation}
The Heisenberg orders of the various blocks are indicated by powers of
$|\eta_0|,$ as we evaluate the symbols along the hyperplane $\eta_0=-1$ to
obtain the model operators. Let $\syprc'_0$ denote the rank one projection, which
is the principal symbol of $\bcS'.$ If $n$ is even, then

\begin{equation}
\esym{eH}{}(\cR^{\prime\even}_-)(-)=
\left(\begin{matrix}\begin{matrix}\bzero\end{matrix}&
\begin{matrix} \bzero
\end{matrix}\\
\begin{matrix}\bzero
\end{matrix}&  \begin{matrix} \Id & 0 \\
0 & 1-\syprc_0'
\end{matrix} 
\end{matrix}\right)
\quad
\esym{eH}{}(\cR^{\prime\odd}_-)(-)=
\left(\begin{matrix}\begin{matrix} \Id 
\end{matrix} &
\begin{matrix} \bzero
\end{matrix}\\
\begin{matrix}\bzero
\end{matrix}&  
\begin{matrix}\bzero & 0\\ 0 &\syprc_0' \end{matrix} \end{matrix}\right).
\end{equation}
If $n$ is odd, then
\begin{equation}
\esym{eH}{}(\cR^{\prime\even}_-)(-)=
\left(\begin{matrix}\begin{matrix}\bzero& 0\\ 0&\syprc_0'\end{matrix}&
\begin{matrix} \bzero
\end{matrix}\\
\begin{matrix}\bzero
\end{matrix}&  \begin{matrix} \Id
\end{matrix} 
\end{matrix}\right)
\quad
\esym{eH}{}(\cR^{\prime\odd}_-)(-)=
\left(\begin{matrix}\begin{matrix} \Id & 0\\
0 &1-\syprc_0'
\end{matrix} &
\begin{matrix} \bzero
\end{matrix}\\
\begin{matrix}\bzero
\end{matrix}&  
\begin{matrix}\bzero\end{matrix} \end{matrix}\right).
\end{equation}

\begin{proposition} If $X$ is strictly pseudoconcave, then at $p\in bX,$ the
  model operators for $\cT^{\eo}_-,$ in the negative  contact direction, are
  given, for $n$ even by,
\begin{equation}
{}^{eH}\sigma(\cT^{\even}_-)(-)=
\left(\begin{matrix} \frac{\ho-(n-1)}{|\eta_0|^2} & \begin{matrix} -\frac{\cD_-}{|\eta_0|}
\end{matrix}
\\
\begin{matrix} \left[\begin{matrix} \Id & 0 \\
0 &1-2\syprc_0'\end{matrix}\right]\frac{\cD_-}{|\eta_0|}
\end{matrix}&  
 \begin{matrix} \bzero & 0 \\
0 &\syprc_0'\end{matrix}\end{matrix}\right)
\label{7.15.5}
\end{equation}
\begin{equation}
{}^{eH}\sigma(\cT^{\odd}_-)(-)=
\left(\begin{matrix} \frac{\ho+(n-1)}{|\eta_0|^2} & \begin{matrix} \frac{\cD_-}{|\eta_0|}
\end{matrix}
\\
\begin{matrix} -\left[\begin{matrix} \Id & 0 \\
0 &1-2\syprc_0'\end{matrix}\right]\frac{\cD_-}{|\eta_0|}
\end{matrix}&  
 \begin{matrix} \bzero & 0 \\
0 &\syprc_0'\end{matrix}\end{matrix}\right)
\label{7.15.6}
\end{equation}
If $n$ is odd, then
\begin{equation}
{}^{eH}\sigma(\cT^{\even}_-)(-)=
\left(\begin{matrix}\begin{matrix} \bzero & 0 \\
0 & \syprc_0'\\
\end{matrix} &\begin{matrix}-\left[\begin{matrix} \Id & 0 \\
0 &1-2\syprc_0'
\end{matrix}\right]\frac{\cD_-}{|\eta_0|}
\end{matrix}
\\
\begin{matrix} \frac{\cD_-}{|\eta_0|}
\end{matrix}&  
\begin{matrix} \frac{\ho+(n-1)}{|\eta_0|^2} \end{matrix} \end{matrix}\right)
\label{7.15.7}
\end{equation}
\begin{equation}
{}^{eH}\sigma(\cT^{\odd}_-)(-)=
\left(\begin{matrix}\begin{matrix}\bzero & 0 \\
0 &  \syprc_0'\\
\end{matrix} &\begin{matrix}\left[\begin{matrix} \Id & 0 \\
0 &1-2\syprc_0'
\end{matrix}\right]\frac{\cD_-}{|\eta_0|}
\end{matrix}
\\
\begin{matrix} -\frac{\cD_-}{|\eta_0|}
\end{matrix}&  
\begin{matrix} \frac{\ho-(n-1)}{|\eta_0|^2} \end{matrix} \end{matrix}\right)
\label{7.15.8}
\end{equation}
\end{proposition}
\begin{proof} 
In the even case the Heisenberg orders of the blocks are
\begin{equation}
\left(\begin{matrix} -2 & -1\\-1&0\end{matrix}\right),
\end{equation}
while in the odd case they are
\begin{equation}
\left(\begin{matrix} 0 & -1\\-1&-2\end{matrix}\right).
\end{equation}
As before, the proposition follows from this observation, the  computations
above, and Proposition~\ref{prop6}.
\end{proof}

This brings us to the main technical result in this paper.
\begin{theorem}\label{thm1} If $X$ is strictly pseudoconvex (pseudoconcave),
  $E\to X$ a compatible complex vector bundle and $\cS'$ ($\bcS'$) a
  generalized (conjugate) Szeg\H o projector, defined by a compatible
  deformation of the almost complex structure on $H$ induced by the embedding
  of $bX$ as the boundary of $X,$ then the operators $\cT^{\eo}_{E+}$
  ($\cT^{\eo}_{E-}$) are graded elliptic elements of the extended Heisenberg
  calculus. If $X$ is pseudoconvex or $X$ is pseudoconcave and $n$ is odd,
  then,as block matrices, the parametrices for $\cT^{\eo}_{E\pm}$ have
  Heisenberg orders
\begin{equation}
\left(\begin{matrix} 0 & 1\\ 1 & 1\end{matrix}\right).
\end{equation}
If $X$ is pseudoconcave and $n$ is even, then,as block matrices, the
parametrices for $\cT^{\eo}_{E-}$ have Heisenberg orders
\begin{equation}
\left(\begin{matrix} 1 & 1\\ 1 & 0\end{matrix}\right).
\end{equation}

\end{theorem}
\begin{proof} Using standard symbolic arguments, to prove the theorem it suffices
  to construct operators $\cU^{\eo}_{\pm}, \cV^{\eo}_{\pm},$ in the
  extended Heisenberg calculus, so that
\begin{equation}
\begin{split}
\cU^{\eo}_{\pm}\cT^{\eo}_{\pm}&=\Id+\sO_{-1,-1}^{eH}\\
\cT^{\eo}_{\pm}\cV^{\eo}_{\pm}&=\Id+\sO_{-1,-1}^{eH}.
\end{split}
\label{7.15.9}
\end{equation}
As usual, this just amounts to the invertibility of the principal symbols. Away
from the positive (negative) Heisenberg face this is clear, as the
operator is classically elliptic of order $0.$ Along the Heisenberg face, the operator is
graded so a little discussion is required. For a graded Heisenberg operator
$\cA,$ denote the matrix of model operators by
$$A=\left(\begin{matrix} A_{11} & A_{12}\\ A_{21} & A_{22}\end{matrix}
\right).$$
The blocks of $A$ have orders either $i+j-4$ or $2-(i+j).$
Suppose the model operators are invertible with inverses given by
$$B=\left(\begin{matrix} B_{11} & B_{12}\\ B_{21} & B_{22}\end{matrix}
\right).$$
The orders of the blocks of $B$ are either $4-(i+j)$ or $(i+j)-2.$ Let
$\cB$ denote an extended Heisenberg operator with principal symbol given by
$B.$ Then, in the first case, we have
\begin{equation}
\cA\cB=\left(\begin{matrix} \Id+\cE_{-1} & \cE_{-1}\\ \cE_0 &
\Id+\cE_{-1}\end{matrix}\right)\quad \cB\cA=\left(\begin{matrix} \Id +\cF_{-1}&
\cF_{0}\\ \cF_{-1} & \Id+\cF_{-1}\end{matrix}\right).
\end{equation}
Here $\cE_{j}, \cF_{j}$ denote operators with the indicated Heisenberg orders. Setting
\begin{equation}
\cB_r=\cB\left(\begin{matrix} \Id & 0\\ -\cE_0 & \Id\end{matrix}\right),\quad
\cB_l=\left(\begin{matrix} \Id & -\cF_0\\ 0 & \Id\end{matrix}\right)\cB
\end{equation}
gives the right and left parametrices called for in equation~\eqref{7.15.9}. A
similar argument works if the orders of $A$ are $(i+j)-2.$ Thus it suffices to
show that the model operators $\esym{eH}{}(\cT^{\eo}_{\pm})(\pm)$ are
invertible, in the graded sense used above. This is done in the next two sections.
\end{proof}
\begin{remark} In the analysis below we show that the order $2$ block in the
  parametrix is absent, hence it is not necessary to correct $\cB$
  with a triangular matrix.
\end{remark}

\section{Invertibility of the model operators with classical Szeg\H o
  projectors} In this section we prove Theorem~\ref{thm1} with the additional
assumption that the principal symbol of $\cS'$ ($\bcS'$) agrees with the
principal symbol, $\sypr_0,$ ($\syprc_0$) of the classical Szeg\H o projector
(conjugate Szeg\H o projector) defined by the CR-structure on $bX.$ In this
case the structure of the model operators is a little simpler. It is not
necessary to assume that the CR-structure on $bX$ is embeddable, as all that we
require are the symbolic identities
\begin{equation}
\sigma_1(\dbarb\cS)=0,\text{ and }
\sigma_1(\dbarb^*\bcS)=0.
\label{7.26.2}
\end{equation}
Note that $\cS_E$ (or $\bcS_E$) are projectors onto sections of
$\Lambda^{\eo}_b\otimes E\restrictedto_{bX}.$ Because the complex structure of
$E$ is compatible with that of $X,$ using the holomorphic frame introduced
in~\eqref{7.26.1}, we see that
\begin{equation}
\sigma(\cS_E)=\sigma(\cS)\otimes \Id_E,\quad
\sigma(\bcS_E)=\sigma(\bcS)\otimes \Id_E.
\end{equation}
Thus we may continue to suppress the explicit dependence on $E.$

The operators  $\{C_j\}$ are called the creation operators and the operators
$\{C_j^*\}$ the annihilation operators.  They satisfy the commutation relations
\begin{equation}
[C_j,C_k]=[C_j^*,C_k^*]=0,\quad
[C_j,C_k^*]=-2\delta_{jk}
\label{7.16.4}
\end{equation}
The operators $\cD_{\pm}$ act on sums of the form
\begin{equation}
\omega=\sum_{k=0}^{n-1}\sum_{I\in\cI_k'} f_I\bomega^I,
\end{equation}
here $\cI_k'$ are increasing multi-indices of length $k.$ The coefficients,
$\{f_I\}$ are sections of the appropriate holomorphic bundle, assumed
trivialized near $p,$ as described in Section~\ref{ss.2}, with vanishing
connection coefficients. We refer to the terms with $|I|=k$ as the terms of
degree $k.$ For an increasing $k$-multi-index $I=1\leq i_1<i_2<\dots<i_k\leq
n-1,$ $\bomega^I$ is defined by
\begin{equation}
\bomega^I=\frac{1}{2^{\frac k2}}d\bz_{i_1}\wedge\dots\wedge d\bz_{i_k}.
\end{equation}

We first describe the relationships among the operators $\sypr_0,
\cD_+, \syprc_0$ and $\cD_-.$ 
\begin{lemma}\label{lem7} Let $\sypr_0$ and $\syprc_0$ be the symbols of the
  classical Szeg\H o projector and conjugate Szeg\H o projector respectively,
  then
\begin{equation}
\left[\begin{matrix} \sypr_0 & 0 \\
0 &\bzero
\end{matrix}\right]\cD_+=0
\text{ and }
\left[\begin{matrix} \bzero & 0 \\
0 &\syprc_0
\end{matrix}\right]\cD_-=0
\end{equation}
\end{lemma}
\begin{proof} The range of $\cD_+$ in degree $0,$ where $\sypr_0$ acts, is
  spanned by expressions of the form
\begin{equation}
\sum_{j=1}^{n-1}C_j f_Ie_j\bomega^{I},\text{ with }|I|=1.
\end{equation}
Taking the adjoint, the first identity in~\eqref{7.26.2} is equivalent to
$\sypr_0C_j=0$ for all $j,$ and the lemma follows in this case. The range of
$\cD_-$ in degree $n-1,$ where $\syprc_0$ acts, is
  spanned by expressions of the form
\begin{equation}
\sum_{j=1}^{n-1}C_j f_I\epsilon_j\bomega^{I},\text{ with }|I|=n-2.
\end{equation}
Once again,~\eqref{7.26.2} implies that $\syprc_0C_j=0;$  the proof of the
lemma is complete.
\end{proof}

This lemma  simplifies the analysis of the  model operators for
$\cT^{\eo}_{\pm}.$  The following lemma is useful in finding their inverses.
\begin{lemma} Let $\Pi_q$ denote projection onto the terms of degree $q,$
\begin{equation}
\Pi_q\omega=\sum_{I\in\cI_q'} f_I\bomega^I.
\end{equation}
The operators $\cD_{\pm}$ satisfy the identities
\begin{equation}
\cD_+^2=\sum_{j=1}^{n-1}C_jC_j^*\otimes\Id+\sum_{q=0}^{n-1}2q\Pi_q,\quad
\cD_-^2=\sum_{j=1}^{n-1}C_jC_j^*\otimes\Id+\sum_{q=0}^{n-1}2(n-1-q)\Pi_q.
\label{7.16.9}
\end{equation}
\end{lemma}
\begin{proof}
In the proof of this lemma we make extensive usage of the following classical
identities, whose verification we leave to the reader.
\begin{lemma} The operators $\{e_j,\epsilon_j\}$ satisfy the following
  relations
\begin{equation}
\begin{split}
e_je_k=-e_ke_j,&\quad \epsilon_j\epsilon_k=-\epsilon_k\epsilon_j\text{ for all }j,k,\\
\epsilon_je_k&=-e_j\epsilon_k\text{ if }j\neq k.
\end{split}
\label{7.16.1}
\end{equation}
For $j=k$ we have
\begin{equation}
\epsilon_je_j\bomega^I=\begin{cases}
\bomega^I&\text{ if }j\in I\\
0&\text{ if }j\notin I\end{cases}\qquad
e_j\epsilon_j\bomega^I=\begin{cases}
\bomega^I&\text{ if }j\notin I\\
0&\text{ if }j\in I\end{cases}
\end{equation}
\end{lemma}
We start with $\cD_+,$ using the lemma we obtain that
\begin{equation}
\begin{split}
\cD_+^2=-&\sum_{j\neq k}\left(\ha[C_j,C_k]e_je_k+\ha[C_j^*,C_k^*]\epsilon_j\epsilon_k-
[C_j,C^*_k]e_j\epsilon_k\right)+\\
&\sum_{j=1}^{n-1}[C_jC_j^*e_j\epsilon_j+C_j^*C_j\epsilon_je_j].
\end{split}
\label{7.16.5}
\end{equation}
It follows from the commutation relations that the sum over $j\neq k$
vanishes. Using~\eqref{7.16.4} we rewrite the second sum as
\begin{equation}
\sum_{j=1}^{n-1}[C_jC_j^*e_j\epsilon_j+(C_jC_j^*+2)\epsilon_je_j].
\label{7.16.6}
\end{equation}
The statement of the lemma follows easily from~\eqref{7.16.6}, and the fact that
\begin{equation}
\sum_{j=1}^{n-1}\epsilon_je_j\bomega^I=|I|\bomega^I.
\label{7.16.8}
\end{equation}
The argument for
$\cD_-$ is quite similar. The analogous sum over $j\neq k$ vanishes and we see
that
\begin{equation}
\begin{split}
\cD_-^2&=\sum_{j=1}^{n-1}[C^*_jC_je_j\epsilon_j+C_jC_j^*\epsilon_je_j]\\
&=\sum_{j=1}^{n-1}[(C_jC_j^*+2)e_j\epsilon_j+C_jC_j^*\epsilon_je_j].
\end{split}
\label{7.16.7}
\end{equation}
The proof is completed as before using
\begin{equation}
\sum_{j=1}^{n-1}e_j\epsilon_j\bomega^I=(n-1-|I|)\bomega^I.
\end{equation}
instead of~\eqref{7.16.8}.
\end{proof}

Before we construct the explicit inverses, we show that
$\esym{eH}{}(\cT^{\eo}_{\pm})(\pm)$ are Fredholm elements (in the graded
sense), in the isotropic algebra. Notice that this is a purely symbolic
statement in the isotropic algebra. The isotropic blocks have orders
\begin{equation}
\left(\begin{matrix} 0 & 1\\ 1& 2\end{matrix} \right)
\end{equation}
on the pseudoconvex side and on the pseudoconcave side if $n$ is odd, and
orders
\begin{equation}
\left(\begin{matrix} 2 & 1\\ 1& 0\end{matrix} \right)
\end{equation}
on the pseudoconcave side if $n$ is even. The leading order part in the
isotropic algebra is independent of the choice of generalized (conjugate)
Szeg\H o projector.  In the former case we can think of the operator as
defining a map from $H^1(\bbR^{n-1};E_1)\oplus H^{2}(\bbR^{n-1}; E_2)$ to
$H^1(\bbR^{n-1};F_1)\oplus H^{0}(\bbR^{n-1}; F_2)$ for appropriate vector
bundles $E_1, E_2, F_1, F_2.$ In the later case the map is from
$H^2(\bbR^{n-1};E_1)\oplus H^{1}(\bbR^{n-1}; E_2)$ to
$H^0(\bbR^{n-1};F_1)\oplus H^{1}(\bbR^{n-1}; F_2).$ It is as maps between these
spaces that the model operators are Fredholm.
\begin{proposition}\label{prop12} The model operators,
  $\esym{eH}{}(\cT^{\eo}_{\pm})(\pm),$ are graded Fredholm elements in the
  isotropic algebra.
\end{proposition}
\begin{proof} As noted above this is a purely symbolic statement in the
  isotropic algebra. It suffices to show that the model operators are
  invertible, by appropriately graded elements of the isotropic algebra, up to
  an error of lower order. Equation~\eqref{7.16.5} shows that
\begin{equation}
[\ho^{-1}\cD_{\pm}]\cD_{\pm}=\cD_{\pm}[\ho^{-1}\cD_{\pm}]=\Id+\sO^{\iso}_{-1}.
\label{7.20.1}
\end{equation}
Here $\sO^{\iso}_j$ is a term of order at most $j$ in the isotropic algebra.
Up to lower order terms, the model operators are
\begin{equation}
\begin{split}
\esym{eH}{}(\cT^{\eo}_+)(+)=\left(\begin{matrix} \bzero & \mp\cD_+\\
\pm\cD_+ & \ho\end{matrix}\right)\\
n\text{ odd }\quad\esym{eH}{}(\cT^{\eo}_-)(-)=\left(\begin{matrix} \bzero & \mp\cD_-\\
\pm\cD_- & \ho\end{matrix}\right)\\
n\text{ even }\quad\esym{eH}{}(\cT^{\eo}_-)(-)=\left(\begin{matrix} \ho & \mp\cD_-\\
\pm\cD_- & \bzero\end{matrix}\right)
\end{split}
\label{7.20.2}
\end{equation}
The isotropic principal symbol of $\ho$ is $|\eta'|^2.$ For these computations,
we let $\ho^{-1}$ denote a model operator with isotropic principal symbol
$|\eta'|^{-2}.$ Using~\eqref{7.20.1}, we see that the operators
in~\eqref{7.20.2} have right parametrices:
\begin{equation}
\begin{split}
\left(\begin{matrix} \bzero & \mp\cD_+\\
\pm\cD_+ & \ho\end{matrix}\right)\left(\begin{matrix} \Id & \pm\ho^{-1}\cD_+\\
\mp\ho^{-1}\cD_+ & \bzero\end{matrix}\right)=\Id+\sO^{\iso}_{-1}\\
n\text{ odd }\quad\left(\begin{matrix} \bzero & \mp\cD_-\\
\pm\cD_- & \ho\end{matrix}\right)\left(\begin{matrix} \Id & \pm\ho^{-1}\cD_-\\
\mp\ho^{-1}\cD_- & \bzero\end{matrix}\right)=\Id+\sO^{\iso}_{-1}\\
n\text{ even }\quad\left(\begin{matrix} \ho & \mp\cD_-\\
\pm\cD_- & \bzero\end{matrix}\right)\left(\begin{matrix} \bzero & \pm\ho^{-1}\cD_-\\
\mp\ho^{-1}\cD_- & \Id\end{matrix}\right)=\Id+\sO^{\iso}_{-1}
\end{split}
\label{7.20.3}
\end{equation}
The same model operators provide left parametrices as well. This proves the proposition.
\end{proof}
\begin{remark} Note that the block of the principal symbols of the parametrices,
  expected to have order 2, actually vanishes.  As a result, the inverses of the
  model operators have Heisenberg order at most $1,$ which in turn allows us to
  deduce the standard subelliptic $\ha$-estimates for these boundary value
  problems.
\end{remark}

The operators $\cD_{\pm}^{\even}$ and $\cD_{\pm}^{\odd}$ are adjoint to one
another. From~\eqref{7.16.9} and the well known properties of the harmonic
oscillator, it is clear that $\cD_+^{\even}\cD_{+}^{\odd}$ is invertible. As
$\cD_+^{\even}$ has a one dimensional null space this easily implies that
$\cD_+^{\odd}$ is injective with image orthogonal to the range
of $\sypr_0,$ while $\cD_+^{\even}$ is surjective. The analogous statements for
$\cD_-^{\eo}$ depend on the parity of $n,$  as $\cD_-^2$ has a null space of
dimension one spanned by the forms of degree $n-1$ in the image of $\syprc_0.$
If $n$ is even, then $\cD_-^{\even}$ is injective and $\cD_-^{\odd}$ is
surjective, with a one dimensional null space spanned by the range of
$\syprc_0.$ If $n$ is odd, then $\cD_-^{\odd}$ is injective and $\cD_-^{\even}$
is surjective. With these observations we easily invert the model operators.

We begin with the $+$ side. Let $[\cD_+^{\even}]^{-1}u$ denote the unique solution
to the equation
$$\cD_+^{\even} v=u,$$
orthogonal to the null space of $\cD_+^{\even}.$ We let 
\begin{equation}
\sh{u}=\left(\begin{matrix} 1-\sypr_0 & 0\\ 0& \Id\end{matrix}\right)u;
\end{equation}
this is the projection onto the range of $\cD_+^{\odd}$ and
\begin{equation}
u_0= \left(\begin{matrix} \sypr_0 & 0\\ 0& \bzero\end{matrix}\right)u,
\end{equation}
denotes the projection onto the nullspace of $\cD_+^{\even}.$ We let
$[\cD_+^{\odd}]^{-1}$ denote the unique solution to
$$\cD_+^{\odd}v=\sh{u}.$$ 
Proposition~\ref{prop12} shows that these
partial inverses are isotropic operators of order $-1.$

With this notation we find the inverse of $\esym{eH}{}(\cT^{\even}_+)(+).$
The vector $[u,v]$ satisfies
\begin{equation}
\esym{eH}{}(\cT^{\even}_+)(+)\left[\begin{matrix} u\\ v
\end{matrix}\right]=\left[\begin{matrix} a\\ b
\end{matrix}\right]
\end{equation}
if and only if
\begin{equation}
\begin{split}
u&=a_0+[\cD_+^{\even}]^{-1}(\ho-(n-1))[\cD_+^{\odd}]^{-1}
\sh{a}+[\cD_+^{\even}]^{-1} b\\
v&=-[\cD_+^{\odd}]^{-1}\sh{a}.
\end{split}
\label{7.19.13}
\end{equation}
Writing out the inverse as a block matrix of operators, with appropriate
factors of $\eta_0$ included, gives:
\begin{equation}
\begin{split}
[\esym{eH}{}&(\cT^{\even}_+)(+)]^{-1}=\\
&\left[\begin{matrix} \left(\begin{matrix} \sypr_0
      &0\\0&\bzero\end{matrix}\right)
+[\cD_+^{\even}]^{-1}(\ho-(n-1))[\cD_+^{\odd}]^{-1}\left(\begin{matrix}1- \sypr_0
      &0\\0&\Id\end{matrix}\right) & \eta_0[\cD_+^{\even}]^{-1}\\
-\eta_0[\cD_+^{\odd}]^{-1}\left(\begin{matrix}1- \sypr_0
      &0\\0&\Id\end{matrix}\right) & \bzero\end{matrix}\right]
\end{split}
\end{equation}
The isotropic operators $[\cD_{+}^{\eo}]^{-1}$ are of order $-1,$ whereas
$[\cD_+^{\even}]^{-1}(\ho-(n-1))[\cD_+^{\odd}]^{-1}$ is of order zero. The
Schwartz kernel of $\sypr_0$ is rapidly decreasing. From this we conclude that
the Heisenberg orders, as a block matrix, of the parametrix for
$[\esym{eH}{}(\cT^{\even}_+)(+)]$  are
\begin{equation}
\left(\begin{matrix} 0 & 1 \\ 1&1\end{matrix}\right).
\label{7.16.10}
\end{equation}
We get a $1$ in the lower right corner because the principal symbol, a priori
of order $2,$ of this entry vanishes. The solution for the odd case is given by
\begin{equation}
\begin{split}
u&=a_0+[\cD_+^{\even}]^{-1}(\ho+(n-1))[\cD_+^{\odd}]^{-1}
\sh{a}-[\cD_+^{\even}]^{-1} b\\
v&=[\cD_+^{\odd}]^{-1}\sh{a}.
\end{split}
\end{equation}
Once again the $2,2$ block of  $[\esym{eH}{}(\cT^{\odd}_+)(+)]^{-1}$ vanishes,
and the principal symbol has the Heisenberg orders indicated
in~\eqref{7.16.10}.

We complete this analysis by writing the solutions to
\begin{equation}
\esym{eH}{}(\cT^{\eo}_-)(-)\left[\begin{matrix} u \\ v\end{matrix}\right]=
\left[\begin{matrix} a \\ b\end{matrix}\right],
\end{equation}
in the various cases. For $n$ even, the operator $\cD_-^{\even}$ is injective and
$\cD_-^{\odd}$ has a one dimensional null space. We let $u_0$ denote the
projection of $u$ onto the null space and $\sh{u}$ the projection onto its
complement. With the notation for the partial inverses of $\cD_-^{\eo}$
analogous to that used in the $+$ case, we have the solution operators:
\begin{equation}
\begin{split}
\text{ even }\quad
u &=[\cD_-^{\even}]^{-1}\sh{b}\\
v&=b_0+[\cD_-^{\odd}]^{-1}(\ho-(n-1))[\cD_-^{\even}]^{-1}\sh{b}-[\cD_-^{\odd}]^{-1}a\\
\text{ odd }\quad
u &=-[\cD_-^{\even}]^{-1}\sh{b}\\
v&=b_0+[\cD_-^{\odd}]^{-1}(\ho-(n-1))[\cD_-^{\even}]^{-1}\sh{b}+[\cD_-^{\odd}]^{-1}a
\end{split}
\label{7.16.11}
\end{equation}
Here and in~\eqref{7.16.12}, ``even'' and ``odd'' refer to the parity of the
spinor.  For $n$ even, the operator $\cD_-^{\odd}$ is injective and
$\cD_-^{\even}$ has a one dimensional null space. We let $u_0$ denote the
projection of $u$ onto the null space and $\sh{u}$ the projection onto its
complement.
\begin{equation}
\begin{split}
\text{ even }\quad
u &=
a_0+[\cD_-^{\even}]^{-1}(\ho+(n-1))[\cD_-^{\odd}]^{-1}\sh{a}+[\cD_-^{\even}]^{-1}
b\\
v &= -[\cD_-^{\odd}]^{-1}\sh{a},\\
\text{ odd }\quad
u &=
a_0+[\cD_-^{\even}]^{-1}(\ho+(n-1))[\cD_-^{\odd}]^{-1}\sh{a}-[\cD_-^{\even}]^{-1}
b\\
v &= [\cD_-^{\odd}]^{-1}\sh{a}.
\end{split}
\label{7.16.12}
\end{equation}
If $n$ is even, then the $(1,1)$ block of the principal symbols of
$[\esym{eH}{}(\cT^{\eo}_-)(-)]^{-1}$ vanishes and therefore the Heisenberg orders
of the blocks of the parametrices are
\begin{equation}
\left[\begin{matrix} 1 & 1\\ 1& 0\end{matrix}\right].
\end{equation}
If $n$ is odd, then the $(2,2)$ block the principal symbols of
$[\esym{eH}{}(\cT^{\eo}_-)(-)]^{-1}$ vanishes and therefore the Heisenberg orders
of the blocks of the parametrices are
\begin{equation}
\left[\begin{matrix} 0 & 1\\ 1& 1\end{matrix}\right].
\end{equation}

For the case of classical Szeg\H o projectors, Lemma~\ref{lem7} implies that
the model operators satisfy
\begin{equation}
[\esym{eH}{}(\cT^{\eo}_{\pm})(\pm)]^*=\esym{eH}{}(\cT^{\ooee}_{\pm})(\pm).
\end{equation}
From Proposition~\ref{prop12} we know that these are Fredholm operators. Since
we have shown that all the operators $\esym{eH}{}(\cT^{\eo}_{\pm})(\pm)$ are
surjective, i.e., have a left inverse, it follows that all are in fact
injective and therefore invertible.  In all cases this completes the proof of
Theorem~\ref{thm1} in the special case that the principal symbols of $\cS'$
or $\bcS'$ agree with those of the classical Szeg\H o projector or conjugate
Szeg\H o projector.

\section{Invertibility of the model operators with generalized\\ Szeg\H o
  projectors} 

The proof of Theorem~\ref{thm1}, with generalized Szeg\H o projectors, is not
much different from that covered in the previous section. We show here that the
parametrices for $\esym{eH}{}(\cT^{\eo}_{-})(-)$ differ from those with
classical Szeg\H o projectors (or conjugate Szeg\H o projectors) by operators
of finite rank. The Schwartz kernels of the correction terms are in the Hermite
ideal, and so do not affect the Heisenberg orders of the blocks in the
parametrix. As before the principal symbol in the $(2,2)$ block (or $(1,1)$ block,
where appropriate) vanishes.

In~\cite{EpsteinMelrose3}  we characterize the set of compatible almost complex
structures in the following way:
\begin{lemma} Let $J_1$ and $J_2$ be compatible almost complex structures on
  the co-oriented contact manifold $Y.$ For each $p\in Y$ there is a Darboux
  coordinate system centered at $p$, so that, if $(\eta_0,\eta')$ are the
  linear coordinates on $T^*_pY,$ then
\begin{equation}
h_{J_1}(\eta)=\sum_{j=1}^{2(n-1)}\eta_j^2\text{ and }
h_{J_2}(\eta)=\sum_{j=1}^{n-1}[\mu_j\eta_j^2+\mu_j^{-1}\eta_{j+n-1}^2]
\end{equation}
for positive numbers $(\mu_1,\dots,\mu_{n-1}).$
\end{lemma}

We split the coordinates $\eta'$ into
$(w_1,\dots,w_{n-1};\varphi_1,\dots,\varphi_{n-1}).$ Let $\ho_1$ and $\ho_2$ denote
the harmonic oscillators obtained by quantizing these symbols with respect to
this splitting, then the ground states for these operators are spanned by
\begin{equation}
\begin{split}
v_{0}{^1}&=m^1\exp\left[-\ha\sum_{j=1}^{n-1}w_j^2\right]\\
v_{0}^{2}&=m^2\exp\left[-\ha\sum_{j=1}^{n-1}\left(\frac{w_j}{\mu_j}\right)^2\right],
\end{split}
\label{7.19.1}
\end{equation}
with $m^j$ chosen so that $\|v_0^j\|_{L^2}=1.$
From these expressions we easily deduce the following result.
\begin{lemma}\label{lem22} If $J_1$ and $J_2$ are compatible almost complex structures,
  then, with respect to the $L^2$-inner product on $\bbR^{n-1}$ defined by a
  choice of splitting of $H_p$, we have
\begin{equation}
\langle v_{0}^1, v_{0}^2\rangle>0.
\end{equation}
On a compact manifold, this inner product is a smooth function, bounded below
by a positive constant. If $\sypr_0^j$ denote the projections onto the respective
vacuum states, then
\begin{equation}
\langle v_{0}^1, v_{0}^2\rangle^2=\Tr\sypr_0^1\sypr_0^2,
\label{7.19.2}
\end{equation}
is therefore well defined independent of the choice of quantization.
\end{lemma}
\begin{proof} Only the second statement requires a proof. In terms of any
  Darboux coordinate system, the projection onto the vacuum state has Schwartz kernel
\begin{equation}
v_0^j\otimes v_0^{jt}.
\label{7.19.5}
\end{equation}
This shows that~\eqref{7.19.2} is correct. It is shown
in~\cite{EpsteinMelrose3} that the trace is independent of the choice of
quantization.
\end{proof}

For our applications, the following corollary is very useful.
\begin{corollary}\label{cor2} Let $J_1$ and $J_2$ be compatible almost complex structures,
In a choice of quantization we define the model operator
\begin{equation}
\relp_{21}=\frac{\sypr_0^2\sypr_0^1}{\Tr \sypr_0^2\sypr_0^1}.
\label{7.19.4}
\end{equation}
This operator is globally defined, belongs to the Hermite ideal,  and satisfies
\begin{equation}
\sypr_0^1\relp_{21}=\sypr_0^1.
\label{7.19.6}
\end{equation}

\end{corollary}
\begin{proof} The first statement follows from Lemma~\ref{lem22} and the
  fact that the symbols of the projectors are globally defined. The relation
  in~\eqref{7.19.4} is easily proved using the representations of $\sypr_0^j$
  given in~\eqref{7.19.5}. The fact that $\relp_{21}$ belongs to the Hermite
  ideal is again immediate from the fact that its Schwartz kernel belongs to
  $\cS(\bbR^{2(n-1)}).$ 
\end{proof}
\begin{remark} The relation~\eqref{7.19.6} implies that
\begin{equation}
\sypr_0^1(\relp_{21}\sypr_0^1-\sypr_0^1)=0.
\label{7.19.7}
\end{equation}
An analogous result, which we use in the sequel, holds for generalized conjugate
Szeg\H o projectors. 
\end{remark}

With these preliminaries, we can now complete the proof of
Theorem~\ref{thm1}. For clarity, we use $\esym{eH}{}(\cT^{\eo}_{\pm})(\pm)$ to
denote the model operators with the classical (conjugate) Szeg\H o projection,
and $\esym{eH}{}(\cT^{\prime\eo}_{\pm})(\pm)$ with a generalized Szeg\H o
projection (or generalized conjugate Szeg\H o projection). 
\begin{proposition} If $\sypr_0'$ ($\syprc_0'$) is a generalized (conjugate) Szeg\H o
  projection, which is a deformation of $\sypr_0,$ ($\syprc_0$), then
  $\esym{eH}{}(\cT^{\eo}_{\pm})(\pm)$ are invertible elements of the isotropic
  algebra. The inverses satisfy
\begin{equation}
[\esym{eH}{}(\cT^{\prime\eo}_{+})(+)]^{-1}=[\esym{eH}{}(\cT^{\eo}_{+})(+)]^{-1}+
\left(\begin{matrix} c_{1} & c_2\\ c_3 & 0\end{matrix}\right),
\label{7.19.8}
\end{equation}
if $n$ is even, then
\begin{equation}
[\esym{eH}{}(\cT^{\prime\eo}_{-})(-)]^{-1}=[\esym{eH}{}(\cT^{\eo}_{-})(-)]^{-1}+
\left(\begin{matrix} 0 & c_2\\ c_3 & c_1\end{matrix}\right),
\label{7.19.9}
\end{equation}
and if $n$ is odd, then
\begin{equation}
[\esym{eH}{}(\cT^{\prime\eo}_{-})(-)]^{-1}=[\esym{eH}{}(\cT^{\eo}_{-})(-)]^{-1}+
\left(\begin{matrix} c_1 & c_2\\ c_3 & 0\end{matrix}\right).
\label{7.19.10}
\end{equation}
Here $c_1, c_2, c_3$ are finite rank operators in the Hermite ideal.
\end{proposition}
\begin{proof}
The arguments for the different cases are very similar. We give the details for
one $+$ case and one $-$ case and formul\ae\ for the answers in representative
cases. In these formul\ae\ we let $z_0$ denote the unit vector spanning the
range of $\sypr_0$ and $z_0',$ the unit vector spanning the range of
$\sypr_0'.$

Proposition~\ref{prop12} implies that $\esym{eH}{}(\cT^{\prime\eo}_{\pm})(\pm)$
are Fredholm operators. Since the differences 
$$\esym{eH}{}(\cT^{\prime\eo}_{\pm})(\pm)-\esym{eH}{}(\cT^{\eo}_{\pm})(\pm)$$
are finite rank operators, it follows that
$\esym{eH}{}(\cT^{\prime\eo}_{\pm})(\pm)$ have index zero. It therefore
suffices to construct a left inverse.

We begin with the $+$ even case by rewriting the equation
\begin{equation}
\esym{eH}{}(\cT^{\prime\even}_{+})(+)\left[\begin{matrix} u\\
    v\end{matrix}\right]
=\left[\begin{matrix} a\\ b\end{matrix}\right],
\end{equation}
as
\begin{equation}
\begin{split}
\left[\begin{matrix}\sypr_0' & 0\\0 &\bzero\end{matrix}\right][u+\cD_+^{\odd} v]&=
\left[\begin{matrix}\sypr_0' & 0\\0 &\bzero\end{matrix}\right]a\\
\left[\begin{matrix}1-\sypr_0' & 0\\0 &\Id\end{matrix}\right]\cD_+^{\odd} v&=
-\left[\begin{matrix}1-\sypr_0' & 0\\0 &\Id\end{matrix}\right]a\\
\cD_+^{\even} u + (\ho-(n-1))v&=b.
\end{split}
\label{7.19.11}
\end{equation}
We solve the middle equation in~\eqref{7.19.11} first. Let
\begin{equation}
\alpha_1=(\frac{z_0'\otimes z_0^t}{\langle z_0',z_0\rangle}-\sypr_0)\Pi_0 a,
\label{7.19.15}
\end{equation}
and note that $\sypr_0\alpha_1=0.$ Corollary~\ref{cor2} shows that this model operator
provides a globally defined symbol. The section $v$ is determined as the unique
solution to
\begin{equation}
\cD_+^{\odd}v=-(\sh{a}-\alpha_1).
\end{equation}
By construction $(1-\sypr_0')(a_0+\alpha_1)=0$ and therefore the second
equation is solved. The section $\sh{u}$ is now uniquely determined by the last
equation in~\eqref{7.19.11}:
\begin{equation}
\sh{u}=[\cD_{+}^{\even}]^{-1}(b+(\ho-(n-1)))[\cD_+^{\odd}]^{-1}(\sh{a}-\alpha_1)).
\end{equation}
This leaves only the first equation, which we rewrite as
\begin{equation}
\left[\begin{matrix}\sypr_0' & 0\\0 &\bzero\end{matrix}\right]u_0=
\left[\begin{matrix}\sypr_0' & 0\\0 &\bzero\end{matrix}\right](a-\cD_+^{\odd}
v-\sh{u}).
\end{equation}
It is immediate that
\begin{equation}
u_0=\frac{z_0\otimes z_0^{\prime t}}{\langle z_0,z_0'\rangle}\Pi_0(a-\cD_+^{\odd}
v-\sh{u}).
\end{equation}
By comparing these equations to those in~\eqref{7.19.13} we see that
$[\esym{eH}{}(\cT^{\prime\even}_{+})(+)]^{-1}$ has the required form. The
finite rank operators are finite sums of terms involving $\sypr_0,$ $z_0\otimes
z_0^{\prime t}$ and $z_0^{\prime t}\otimes z_0,$ and are therefore in the
Hermite ideal.

The solution in the $+$ odd case is given by
\begin{equation}
\begin{split}
v&=[\cD_+^{\odd}]^{-1}(\sh{a}-\alpha_1)\\
\sh{u} &=[\cD_+^{\even}]^{-1}[(\ho+(n-1))v-b]\\
u_0&=\frac{z_0\otimes z_0^{\prime t}}{\langle z_0,z_0'\rangle}\Pi_0(a+\cD_+^{\odd}
v-\sh{u})
\end{split}
\label{7.19.14}
\end{equation}
As before $\alpha_1$ is given by~\eqref{7.19.15}. Again the inverse of
$\esym{eH}{}(\cT^{\prime\odd}_{+})(+)$ has the desired form.

In the $-$ case, the computations are nearly identical for $n$ odd. We leave
the details to the reader, and conclude by providing the solution for $n$
even. We let $\bz_0$ and $\bz_0'$ denote unit vectors spanning the ranges of
$\syprc_0$ and $\syprc_0'$ respectively. We let
\begin{equation}
\beta_1=(\frac{\bz_0'\otimes \bz_0^t}{\langle
  \bz_0',\bz_0\rangle}-\syprc_0)\Pi_{n-1}b
\label{7.19.16}
\end{equation}

The solution to
\begin{equation} 
[\esym{eH}{}(\cT^{\prime\even}_{-})(-)]^{-1}\left[\begin{matrix} u\\
    v\end{matrix}\right]
=\left[\begin{matrix} a\\ b\end{matrix}\right],
\end{equation}
is given by
\begin{equation}
\begin{split}
u &=[\cD_-^{\even}]^{-1}(\sh{b}-\beta_1)\\
\sh{v}&=[\cD_-^{\odd}]^{-1}((\ho-(n-1))u-a)\\
v_0&=\frac{\bz_0\otimes \bz_0^{\prime t}}{\langle
  \bz_0,\bz_0'\rangle}\Pi_{n-1}
(b-\cD_-^{\even}u-\sh{v}).
\end{split}
\end{equation}
The result for $\cT^{\prime\odd}_-$ is
\begin{equation}
\begin{split}
u &=-[\cD_-^{\even}]^{-1}(\sh{b}-\beta_1)\\
\sh{v}&=[\cD_-^{\odd}]^{-1}(a-(\ho-(n-1))u)\\
v_0&=\frac{\bz_0\otimes \bz_0^{\prime t}}{\langle
  \bz_0,\bz_0'\rangle}\Pi_{n-1}
(b+\cD_-^{\even}u-\sh{v}).
\end{split}
\end{equation}
We leave the computations in the case of $n$ odd to the reader. In all cases we
see that the parametrices have the desired grading and this completes the proof of
the proposition.
\end{proof}

As noted above, the operators $\esym{eH}{}(\cT^{\prime\eo}_{\pm})(\pm)$ are
Fredholm operators of index zero. Hence, Solvability of the equations
\begin{equation}
\esym{eH}{}(\cT^{\prime\eo}_{\pm})(\pm) \left[\begin{matrix} u\\
    v\end{matrix}\right]
=\left[\begin{matrix} a\\ b\end{matrix}\right],
\end{equation}
for all $[ a, b]$ implies the uniqueness and therefore the invertibility of the
model operators. This completes the proof of Theorem~\ref{thm1}. We now turn to
applications of these results.

\section{The Fredholm property}
Let $\cD$ be a (pseudo)differential operator acting on smooth sections of $F\to
X,$ and $\cB$ a (pseudodifferential) boundary operator acting on sections of
$F\restrictedto_{bX}.$ The pair $(\cD,\cB)$ is the densely defined operator,
$\sigma\mapsto \cD\sigma,$ acting on sections of $F,$ smooth on $\bX,$ that satisfy
\begin{equation}
\cB[\sigma]_{bX}=0.
\end{equation}
The notation $\overline{(\cD,\cB)}$ is the closure of $(\cD,\cB)$ in the graph
norm
\begin{equation}
\|\sigma\|_{\cD}^2=\|\cD\sigma\|_{L^2}^2+\|\sigma\|_{L^2}^2.
\end{equation}
We let $H_{\cD}$ denote the domain of the closure, with norm defined by
$\|\cdot\|_{\cD}.$ The following general result about Dirac operators, proved
in~\cite{BBW}, is useful for our analysis:
\begin{proposition}\label{prop_bbw} Let $X$ be a compact manifold with boundary
  and $\cD$ an operator of Dirac type acting on sections of $F\to X.$ The trace
  map from smooth sections of $F$ to sections of $F\restrictedto_{bX},$
$$\sigma\mapsto \sigma\restrictedto_{bX},$$
extends to define a continuous map from $H_{\cD}$ to
$H^{-\ha}(bX;F\restrictedto_{bX}).$
\end{proposition}

The results of the previous sections show that the operators
$\cT^{\prime\eo}_{\pm}$ are elliptic elements in the extended Heisenberg
calculus. We now let $\cU^{\prime\eo}_{\pm}$ denote a left and right parametrix
defined so that
\begin{equation}
\begin{split}
&\cU^{\prime\eo}_{\pm}\cT^{\prime\eo}_{\pm}=\Id+K_1\\
&\cT^{\prime\eo}_{\pm}\cU^{\prime\eo}_{\pm}=\Id+K_2,
\end{split}
\label{7.22.1}
\end{equation}
with $K_1, K_2$ finite rank smoothing operators.  The principal symbol
computations show that $\cU^{\prime\eo}_{\pm}$ has classical order $0$ and
Heisenberg order at most $1.$  Such an operator defines a bounded map from
$H^{\ha}(bX)$ to $L^2(bX).$ This follows because such  operators are
contained in $\Psi_{eH}^{\ha,1,1}.$ If $\Delta$ is a positive (elliptic) Laplace
operator, then $\cL=(\Delta+1)^{\frac 14}$ lifts to define an invertible
elliptic element of this operator class. An operator $A\in\Psi_{eH}^{\ha,1,1}$
can be expressed in the form
\begin{equation}
A=A'\cL\text{ where } A'\in\Psi_{eH}^{0,0,0}.
\end{equation}
It is shown in~\cite{EpsteinMelrose3}, that operators in $\Psi_{eH}^{0,0,0}$
act boundedly on $H^s,$ for all real $s.$ This proves the following result:
\begin{proposition}\label{prop15} The operators $\cU^{\prime\eo}_{\pm}$ define bounded maps
  from $H^{s}(bX;F)$ to $H^{s-\ha}(bX;F)$ for $s\in\bbR.$ Here $F$ is an
  appropriate vector bundle over $bX.$
\end{proposition}
\begin{remark} Various similar results appear in the literature, for
  example in~\cite{Greiner-Stein1} and~\cite{Beals-Stanton}. While the simple
  result in the proposition is adequate for our purposes, much more precise,
  anisotropic estimates can also be deduced.
\end{remark}

The mapping properties of the boundary parametrices allow us to show that the
graph closures of the operators $(\eth^{\eo}_{\pm},\cR^{\prime\eo}_{\pm})$ are
Fredholm. As usual $E\to X$ is a compatible complex vector bundle. Except when
needed for clarity, the explicit dependence on $E$ is suppressed. 
\begin{theorem}\label{thm2} Let $X$ be a strictly pseudoconvex (pseudoconcave)
  manifold. The graph closures of $(\eth^{\eo}_{E+},\cR^{\prime\eo}_{E+}),$  
($(\eth^{\eo}_{E-},\cR^{\prime\eo}_{E-})$), respectively, are Fredholm operators.
\end{theorem}
\begin{proof} The argument is formally identical for all the different cases, so we do
  just the case of $(\eth^{\even}_+,\cR^{\prime\even}_+).$ As before $Q^{\even}$ is a
  fundamental solution for $\eth^{\even}_+$ and $\cK$ is the Poisson kernel
  mapping the range of $\cP^{\even}_+$ into the null space of
  $\eth^{\even}_{+}.$ We need to show that the range of the closure is closed,
  of finite codimension, and that the null space is finite dimensional. 

Let $f$ be an $L^2$-section of $\Lambda^{\odd}\otimes E;$ with
\begin{equation}
u_1=Q^{\even} f\text{ and
}u_0=-\cK\cU^{\prime\even}_+\cR^{\prime\even}_+[u_1]_{bX},
\end{equation} 
we let $u=u_0+u_1.$ Proposition~\ref{prop15} and standard estimates imply that,
for $s\geq 0,$ there are constants $C_{s1} ,C_{s2},$ independent of $f,$ so
that
\begin{equation}
\|u_1\|_{H^{s+1}}\leq C_{s1}\|f\|_{H^{s}},
\quad
\|u_0\|_{H^{s+\ha}}\leq C_{s2}\|f\|_{H^s}.
\label{7.22.4}
\end{equation}

The crux of the matter is to show that $\cR^{\prime\even}_+[u_0+u_1]_{bX}=0.$
For data satisfying finitely many linear conditions, this is a consequence of
the following lemma.
\begin{lemma}\label{lemm77} If $\cT^{\prime\even}_+ v\in\Im\cR^{\prime\even}_+,$ then
\begin{equation}
\cT^{\prime\even}_{+}\cP^{\even}_+ v=\cT^{\prime\even}_{+}v.
\end{equation}
\end{lemma}
\begin{proof}[Proof of the lemma] As
  $(\Id-\cR^{\prime\even}_+)\cT^{\prime\even}_{+}=
 \cT^{\prime\even}_{+}(\Id-\cP^{\even}_+)$ we see that the hypothesis of the
 lemma implies that
\begin{equation}
\cT^{\prime\even}_{+}(\Id-\cP^{\even}_+)v=(\Id-\cR^{\prime\even}_+)\cT^{\prime\even}_{+}v=0.
\end{equation}
The conclusion follows from this relation.
\end{proof}

Since $u_0\in\Ker\eth_+^{\even}$ it follows that
$(\Id-\cP^{\even}_+)[u_0]_{bX}=0,$ 
and therefore the definition of $u_0$ implies that:
\begin{equation}
\begin{split}
\cR^{\prime\even}_{+}[u_0+u_1]_{bX}
&=\cT^{\prime\even}_{+}[u_0]_{bX}+\cR^{\prime\even}_+[u_1]_{bX}\\
&=-\cT^{\prime\even}_+\cP^{\even}_+\cU^{\prime\even}_+\cR^{\prime\even}_+[u_1]_{bX}
+\cR^{\prime\even}_+[u_1]_{bX}.
\end{split}
\label{7.22.55}
\end{equation}
If
\begin{equation}
K_2\cR^{\prime\even}_+[u_1]_{bX}=K_2\cR^{\prime\even}_+[Q^{\even}f]_{bX}=0,
\label{7.22.6}
\end{equation}
then
\begin{equation}
\cT^{\prime\even}_+\cU^{\prime\even}_+\cR^{\prime\even}_+[u_1]_{bX}
=\cR^{\prime\even}_+[u_1]_{bX}\in\Im\cR^{\prime\even}_+.
\end{equation} 
Hence, applying Lemma~\ref{lemm77}, we see that
\begin{equation}
\begin{split}
\cT^{\prime\even}_+\cP^{\even}_+\cU^{\prime\even}_+\cR^{\prime\even}_+[u_1]_{bX}&=
\cT^{\prime\even}_+\cU^{\prime\even}_+\cR^{\prime\even}_+[u_1]_{bX}\\
&=\cR^{\prime\even}_+[u_1]_{bX}
\end{split}
\label{7.22.5}
\end{equation}
Combining~\eqref{7.22.55} and~\eqref{7.22.5} gives the desired result:
\begin{equation}
\cR^{\prime\even}_{+}[u_0+u_1]_{bX}=0.
\end{equation}
It is
also clear that, if $f\in H^s,$  then
$u\in H^{s+\ha}.$ In particular, if $f$ is smooth, then so is $u.$ Hence $u$
belongs to the domain of $(\eth^{\even}_+,\cR^{\prime\even}_+).$ 

The operator $K_2$ is a finite rank smoothing operator, and therefore the
composition 
\begin{equation}
f\mapsto K_2\cR^{\prime\even}_+[Q^e f]_{bX}
\end{equation}
has a kernel of the form
\begin{equation}
\sum_{j=1}^M u_j(x)v_j(y)\text{ for }(x,y)\in bX\times X,
\end{equation}
with
$$u_j\in\CI(bX) \text{ and }v_j\in\CI(\bX).$$
Hence, an $L^2$-section, $f$
satisfying~\eqref{7.22.6} can be obtained as the limit of a sequence of smooth
sections $<f^n>$ that also satisfy this condition. Let $<u^n>$ be the smooth
solutions to
\begin{equation}
\eth^{\even}_+u^n=f^n,\quad \cR^{\prime\even}_+[u^n]_{bX}=0,
\end{equation} 
constructed above.  The estimates in~\eqref{7.22.4} show that
$<u^n>$ converges to a limit $u$ in $H^{\ha}.$ It is also clear that
$\eth^{\even}_+u^n$ converges weakly to $\eth^{\even}_+u,$ and in $L^2$ to $f.$
Therefore $<u^n>$ converges to $u$ in the graph norm. This shows that $u$ is in
the domain of the closure and satisfies $\eth^{\even}_+u=f.$ As the composition
$$f\mapsto K_2\cR^{\prime\even}_+[Q^{\even} f]_{bX},$$
is bounded, it follows that the range of 
$\overline{(\eth^{\even}_+,\cR^{\prime\even}_+)}$ contains a closed subspace of finite
codimension and is therefore also a closed subspace of finite codimension. 

To complete the proof of the theorem we need to show that the null space is
finite dimensional. Suppose that $u$ belongs to the null space of
$\overline{(\eth^{\even}_+,\cR^{\prime\even}_+)}.$ This implies that there is a
sequence of smooth sections $<u^n>$ in the domain of the operator, converging
to $u$ in the graph norm, such that $\|\eth^{\even}_+u^n\|_{L^2}$ converges to
zero. Hence $\eth^{\even}_+ u=0$ in the weak sense. Proposition~\ref{prop_bbw}
shows that $u$ has boundary values in $H^{-\ha}(bX)$ and that, in the sense of
distributions,
$$\cR^{\prime\even}_+[u]_{bX}=\lim_{n\to\infty}\cR^{\prime\even}_+[u^n]_{bX}=0.$$
Since $u$ is in the null space of $\eth^{\even}_+,$ it is also the
case that $\cP^{\even}_+[u]_{bX}=[u]_{bX}.$ These two facts imply that
$\cT^{\prime\even}_+[u]_{bX}=0.$ Composing on the left with
$\cU^{\prime\even}_{+}$ shows that
\begin{equation}
(\Id+K_1)[u]_{bX}=0.
\end{equation}
As $K_1$ is a finite rank smoothing operator, we conclude that $[u]_{bX}$ and
therefore $u$ are smooth. By the unique continuation property for Dirac
operators, the dimension of the null space of
$\overline{(\eth^{\even}_{+},\cR^{\prime\even}_+)}$ is bounded by the dimension
of the null space of $(\Id+K_1).$ This completes the proof of the assertion
that $\overline{(\eth^{\even}_+,\cR^{\prime\even}_+)}$ is a Fredholm operator. The
proofs in the other cases, are up to minor changes in notation, identical.
\end{proof}
\begin{remark} In the proof of the theorem we have constructed right parametrices
  $\cQ^{\prime\eo}_{\pm}$ for the boundary value problems
  $(\eth^{\eo}_{\pm},\cR^{\prime\eo}_{\pm}),$ which gain a half a derivative.
\end{remark}
  
We close this section with Sobolev space estimates for the operators
$(\eth^{\eo}_{\pm},\cR^{\prime\eo}_{\pm}).$
\begin{theorem}\label{thm33} Let $X$ be a strictly pseudoconvex (pseudoconcave)
  manifold, and $E\to X$ a compatible complex vector bundle. For each $s\geq 0,$
  there is a positive constant $C_s$ such that if $u$ is an $L^2$-solution to
$$\eth^{\eo}_{E\pm} u = f\in H^s(X)\text{ and }\cR^{\prime\eo}_{E\pm}[u]_{bX}=0$$
in the sense of distributions, then
\begin{equation}
\|u\|_{H^{s+\ha}}\leq C_s[\|\eth^{\eo}_{E\pm} u\|_{H^s}+\|u\|_{L^2}].
\label{7.22.100}
\end{equation}
\end{theorem}
\begin{proof} With $u_1=Q^{\eo} f,$ we see that $u_1\in H^{s+1}(X)$ and
$$\eth^{\eo}_{\pm} (u-u_1)=0\text{ with }\cR^{\prime\eo}_{\pm}[u-u_1]_{bX}=
-\cR^{\prime\eo}_{\pm}[u_1]_{bX}.
$$
These relations imply that $\cP^{\eo}_{\pm}[u-u_1]_{bX}=[u-u_1]_{bX}$ and
therefore
\begin{equation}
-\cR^{\prime\eo}_{\pm}[u_1]_{bX}=\cR^{\prime\eo}_{\pm}[u-u_1]_{bX}=
\cT^{\prime\eo}_{\pm}[u-u_1]_{bX}.
\end{equation}

We apply $\cU^{\prime\eo}_{\pm}$ to this equation to deduce that
\begin{equation}
(\Id+K_1)[u-u_1]_{bX}=-\cU^{\prime\eo}_{\pm}\cR^{\prime\eo}_{\pm}[u_1]_{bX}.
\end{equation}
Because $K_1$ is a smoothing operator, Proposition~\ref{prop15} implies that there
is a constant $C_s'.$ so that
\begin{equation}
\|[u-u_1]_{bX}\|_{H^s(bX)}\leq
C_s'[\|u_1\|_{H^{s+\ha}(bX)}+\|[u-u_1]\|_{H^{-\ha}(bX)}].
\label{7.22.101}
\end{equation}
As the Poisson kernel carries $H^{s}(bX)$ to $H^{s+\ha}(X),$ boundedly, this
estimate shows that $u=u-u_1+u_1$ belongs to $H^{s+\ha}(X)$ and that there is a
constant $C_s$ so that
\begin{equation}
\|u\|_{H^{s+\ha}}\leq C_s[\|f\|_{H^s}+\|u\|_{L^2}]
\end{equation}
This proves the theorem.
\end{proof}
\begin{remark} In the case $s=0,$ this proof gives a slightly better result:
  the Poisson kernel actually maps $L^2(bX)$ into $H_{\hn}(X)$ and therefore
  the argument shows that there is a constant $C_0$ such that if $u\in L^2,$
  $\eth^{\eo}_{\pm} u\in L^2$ and $\cR^{\prime\eo}_{\pm} [u]_{bX}=0,$ then
\begin{equation}
\|u\|_{\hn}\leq C_0[\|f\|_{L^2}+\|u\|_{L^2}]
\end{equation}
This is just  the standard $\ha$-estimate for the operators
$\overline{(\eth^{\eo}_{\pm},\cR^{\prime\eo}_{\pm})}$ 
\end{remark}

It is also possible to prove localized versions of these results. The higher norm
estimates have the same consequences as for the $\dbar$-Neumann
problem. Indeed, under certain hypotheses these estimates imply higher norm
estimates for the second order operators considered in~\cite{Epstein4}. We
prove these in the next section after showing the the closures of the formal
adjoints of $(\eth^{\eo}_{\pm},\cR^{\prime\eo}_{\pm})$ are the $L^2$-adjoints
of these operators.

\section{Adjoints of the Spin${}_{\bbC}$ Dirac operators}

In the previous section we proved that the operators
$\overline{(\eth^{\eo}_{\pm},\cR^{\prime\eo}_{\pm})}$ are Fredholm operators, as well as
estimates that they satisfy. In this section we show that the $L^2$-adjoints of
these operators are the closures of the formal adjoints.
\begin{theorem}\label{thm3} If $X$ is strictly pseudoconvex (pseudoconcave),
  $E\to X$ a compatible complex vector bundle, then we have
  the following relations:
\begin{equation}
(\eth^{\eo}_{E\pm},\cR^{\prime\eo}_{E\pm})^*=
\overline{(\eth^{\ooee}_{E\pm},\cR^{\prime\ooee}_{E\pm})}.
\label{7.21.2}
\end{equation}
We take $+$ if $X$ is pseudoconvex and $-$ if $X$ is pseudoconcave. 
\end{theorem}
\begin{proof} The argument follows a standard outline. It is clear that
\begin{equation}
\overline{(\eth^{\ooee}_{\pm},\cR^{\prime\ooee}_{\pm})}\subset
(\eth^{\eo}_{\pm},\cR^{\prime\eo}_{\pm})^*
\label{7.21.3}
\end{equation}
Suppose that the containment is proper. This would imply that, for any nonzero,
real $\mu$ there 
exists a nonzero section
$v\in\Dom_{L^2}((\eth^{\eo}_{\pm},\cR^{\prime\eo}_{\pm})^*),$ such that, for all
$w\in\Dom((\eth^{\ooee}_{\pm},\cR^{\prime\ooee}_{\pm})),$
\begin{equation}
\langle [\eth^{\eo}_{\pm}]^* v,\eth^{\ooee}_{\pm} w\rangle+\mu^2\langle v,w\rangle=0.
\label{7.21.4}
\end{equation}
Suppose that $\cR^{\prime\eo}_{\pm}\eth^{\ooee}_{\pm} w\restrictedto_{bX}=0.$ Since
$v$ belongs to $\Dom_{L^2}((\eth^{\eo}_{\pm},\cR^{\prime\eo}_{\pm})^*)),$ we can
integrate by parts to obtain that
\begin{equation}
\langle v, (\eth^{\eo}_{\pm}\eth^{\ooee}_{\pm}+\mu^2)w\rangle=0.
\label{7.21.5}
\end{equation}

This reduces the proof of the theorem to the following proposition.
\begin{proposition}\label{prop14}
For any nonzero real number $\mu,$ and 
$$f\in\CI(\bX;\Spn^{\ooee}\otimes E),$$
there is a section $w\in\CI(\bX;\Spn^{\ooee}\otimes E),$ which satisfies
\begin{equation}
\begin{split}
(\eth^{\eo}_{\pm}&\eth^{\ooee}_{\pm}+\mu^2)w=f\\
\cR^{\prime\ooee}_{\pm} w\restrictedto_{bX}=0&\quad\text{ and }\quad
\cR^{\prime\eo}_{\pm}\eth^{\ooee}_{\pm} w\restrictedto_{bX}=0.
\end{split}
\label{7.21.6}
\end{equation}
\end{proposition}

Before proving the proposition, we show how it implies the theorem. Let $w,f$
be as in~\eqref{7.21.6}. The boundary conditions satisfied by
$\eth^{\ooee}_{\pm} w$ and ~\eqref{7.21.5} imply that we have
\begin{equation}
\langle v,f\rangle=0.
\end{equation}
As $f\in\CI(\bX;\Spn^{\ooee}\otimes E)$ is arbitrary, this shows that $v=0$ as well and
thereby completes the proof of the theorem. 
\end{proof}

The proposition is a consequence of Theorems~\ref{thm1} and~\ref{thm33}.
\begin{proof}[Proof of Proposition~\ref{prop14}]
The first step is to show that~\eqref{7.21.6} has a weak solution for any
non-zero real number $\mu,$ after which, we use a small extension of
Theorem~\ref{thm33} to show that this solution is actually in
$\CI(\bX;\Spn^{\ooee}\otimes E).$ 
\begin{lemma} Let $Q(w)=\langle\eth^{\ooee}_{\pm}w,\eth^{\ooee}_{\pm}w\rangle,$
  denote the non-negative, symmetric quadratic form with domain:
  \begin{equation}
    \label{eq:100.1}
    \Dom(Q)=\{w\in L^2(X):\: \eth^{\ooee}_{\pm}w\in L^2(X)\text{ and
    }\cR^{\prime\ooee}_{\pm} w\restrictedto_{bX}=0\}. 
  \end{equation}
  The form $Q$ is closed and densely defined.  Let $L$ denote the self adjoint
  operator defined by $Q.$ If $w\in\Dom(L),$ then
\begin{equation}
  \label{eq:100.2}
  \eth^{\eo}_{\pm}\eth^{\ooee}_{\pm}w\in L^2 \text{ and }
  \cR^{\prime\eo}_{\pm}\eth^{\ooee}_{\pm}w\restrictedto_{bX}=0. 
\end{equation}
\end{lemma}
\begin{remark} That a densely defined, closed, symmetric, non-negative
  quadratic form defines a self adjoint operator is the content of Theorem
  VI.2.6 in~\cite{Kato}. For the remainder of this section we let $\rho$ denote
  a defining function for $bX,$ with $d\rho$ of unit length along the sets
  $\{\rho=\epsilon\},$ for $\epsilon$ sufficiently small.
\end{remark}
\begin{proof}[Proof of Lemma] It is clear that $Q$ is densely defined. That the
  form is closed is an immediate consequence of Proposition~\ref{prop_bbw}. By
  definition, the domain of $L$ consists of sections $w\in\Dom(Q),$ such that
  there exists a $g\in L^2,$ for which
  \begin{equation}
    \label{eq:100.3}
    Q(w,v)=\langle g,v\rangle
  \end{equation}
for all $v\in \Dom(Q).$ Since all smooth sections with compact support lie in
$\Dom(Q),$ it follows from~\eqref{eq:100.3} that 
\begin{equation}
  \label{eq:100.4}
  \eth^{\eo}_{\pm}\eth^{\ooee}_{\pm}w=g\in L^2,
\end{equation}
where the operator, $\eth^{\eo}_{\pm}\eth^{\ooee}_{\pm},$ is applied in the
distributional sense. This in turn implies that $w\in H^2_{\loc}(X),$ and that
$\eth^{\ooee}_{\pm}w$ has restrictions to the sets $\{\rho=\epsilon\},$ which
depend continuously on $\epsilon$ in the $H^{-\frac 12}(bX)$-topology.

Now let $v$ be a section, smooth in the closure of $X,$ though not necessarily
in $\Dom(Q).$ The regularity properties of $w$ imply that
\begin{equation}
  \label{eq:100.5}
  Q(w,v)=\langle\eth^{\eo}_{\pm}\eth^{\ooee}_{\pm}w,v\rangle+
\langle\eth^{\ooee}_{\pm}w,\sigma(\eth^{\ooee}_{\pm},-id\rho)v\rangle_{bX}.
\end{equation}
If $v\in\Dom(Q),$ then~\eqref{eq:100.4} shows that the boundary term
in~\eqref{eq:100.5} must vanish. If $h$ is any smooth section defined on $bX,$
then by smoothly extending $(\Id-\cR^{\prime\ooee}_{\pm})h$
to $\bX$ we obtain a smooth section $v\in\Dom(Q),$ with
\begin{equation}
  v\restrictedto_{bX}=(\Id-\cR^{\prime\ooee}_{\pm})h.
\end{equation}
Hence, if $w\in \Dom(L),$ then, for any smooth section $h,$ we have
\begin{equation}
  \label{eq:100.6}
\begin{split}
 \langle \cR^{\prime\eo}_{\pm}\eth^{\ooee}_{\pm}w,
h)\rangle_{bX}&= \langle\sigma(\eth^{\eo}_{\pm},-id\rho)\eth^{\ooee}_{\pm}w,
(\Id-\cR^{\prime\ooee}_{\pm})h)\rangle_{bX}\\
&=0,
\end{split}
\end{equation}
verifying the final assertion of the lemma.
\end{proof}

The operator $L$ is non-negative and self adjoint. Hence for any real $\mu\neq
0,$ and $f\in\CI(\bX;\Spn^{\ooee}\otimes E),$ there is a unique $w\in\Dom(L)$
satisfying~\eqref{7.21.6} in the sense of distributions. To complete the proof
of the proposition we need to show that this solution is smooth.

We rewrite this in terms of the system of first order equations:
\begin{equation}
\begin{split}
&\cD^\mu_{\pm}\left(\begin{matrix} u \\ v\end{matrix}\right)\overset{d}{=}
\left(\begin{matrix} \eth^{\ooee}_{\pm} & -\mu\\
\mu &\eth^{\eo}_{\pm}\end{matrix}\right)\left(\begin{matrix} u \\
  v\end{matrix}\right)=
\left(\begin{matrix} a \\ b\end{matrix}\right),\\
&\cR_{\pm}\left(\begin{matrix} u \\ v\end{matrix}\right)\overset{d}{=}
\left(\begin{matrix} \cR^{\prime\ooee}_{\pm} & 0\\
0 &\cR^{\prime\eo}_{\pm}\end{matrix}\right)\left(\begin{matrix} u \\
  v\end{matrix}\right)_{bX}=0.
\end{split}
\label{7.21.7}
\end{equation}
Clearly the solution constructed above satisfies
\begin{equation}
  \label{eq:100.7}
  \cD^\mu_{\pm}\left(\begin{matrix}
      w\\\frac{1}{\mu}\eth^{\ooee}_{\pm}w\end{matrix}\right)=
\left(\begin{matrix} 0\\\frac{f}{\mu}\end{matrix}\right)\text{ and }
\cR_{\pm}\left(\begin{matrix}
      w\\\frac{1}{\mu}\eth^{\ooee}_{\pm}w\end{matrix}\right)
\restrictedto_{bX}=0,
\end{equation}
in the sense of distributions.
To complete the proof of the proposition it suffices to establish a regularity
result for $( \cD^\mu_{\pm},\cR_{\pm})$ analogous to
Theorem~\ref{thm33}. Indeed essentially the same argument applies to this case.

Let $\cP^{\mu}_{\pm}$ denote the Calderon projector
for the operator $\cD^{\mu}_{\pm},$ and set
\begin{equation}
\cT^{\mu}_{\pm}=\cR_{\pm}\cP^{\mu}_{\pm}+(\Id-\cR_{\pm})(\Id-\cP^{\mu}_{\pm}).
\end{equation}
Theorem~\ref{thm1}  implies that $\cT^{0}_{\pm}$ is a graded elliptic element of
the extended Heisenberg calculus. Let $\cU^{0}_{\pm}$ denote a
parametrix for $\cT^{0}_{\pm}.$ We now show that
\begin{equation}
\cT^{\mu}_{\pm}=\cT^{0}_{\pm}+\sO^{eH}_{-1,-2}
\label{7.21.9}
\end{equation}
Here $\sO^{eH}_{-1,-2}$ is an extended Heisenberg operator,
having Heisenberg order  $-2$ on the appropriate parabolic face and classical
order $-1.$ As the extended Heisenberg order of 
$\cU^0_{\pm}$ is $(0,1)$ we see that this operator is also a parametrix for
$\cT^{\mu}_{\pm}.$  We now verify~\eqref{7.21.9}.

The operator $\cD^{\mu}_{\pm}[\cD^{\mu}_{\pm}]^*$ is given by
\begin{equation}
\cD^{\mu}_{\pm}[\cD^{\mu}_{\pm}]^*=\left(\begin{matrix}
  \eth^{\ooee}_{\pm}\eth^{\eo}_{\pm}+\mu^2 & 0\\
0 & \eth^{\eo}_{\pm}\eth^{\ooee}_{\pm}+\mu^2\end{matrix}\right).
\end{equation}
The fundamental solution $\cQ^{\mu(2)}_{\pm}$ has the form
\begin{equation}
\cQ^{\mu(2)}_{\pm}=\left(\begin{matrix}
  Q^{\eo(2)\mu}_{\pm} & 0\\
0 & Q^{\ooee(2)\mu}_{\pm }\end{matrix}\right),
\end{equation}
where $Q^{\eo(2)\mu}_{\pm}= (\eth^{\ooee}_{\pm}\eth^{\eo}_{\pm}+\mu^2)^{-1}.$
A fundamental solution for $\cD^{\mu}_{\pm}$ is then given by
\begin{equation}
\cQ^{\mu}_{\pm}=[\cD^{\mu}_{\pm}]^*\cQ^{\mu(2)}_{\pm}
=\left(\begin{matrix}\eth^{\eo}_{\pm}Q^{\eo(2)\mu}_{\pm} &\mu
  Q^{\ooee(2)\mu}_{\pm}\\-\mu
  Q^{\eo(2)\mu}_{\pm}&\eth^{\ooee}_{\pm}Q^{\ooee(2)\mu}_{\pm}
\end{matrix}\right).
\label{7.21.11}
\end{equation}

The claim in~\eqref{7.21.9} follows from the observation that
\begin{equation}
Q^{\eo(2)\mu}_{\pm}-Q^{\eo(2)0}_{\pm}\in\sO_{-4},
\label{7.21.10}
\end{equation}
which is a consequence of the resolvent identity
\begin{equation}
(\eth^{\ooee}_{\pm}\eth^{\eo}_{\pm}+\mu^2)^{-1}-
(\eth^{\ooee}_{\pm}\eth^{\eo})^{-1}=
-\mu^2(\eth^{\ooee}_{\pm}\eth^{\eo}_{\pm}+\mu^2)^{-1}(\eth^{\ooee}_{\pm}\eth^{\eo}_{\pm})^{-1},
\end{equation}
and the fact that $\eth^{\ooee}_{\pm}\eth^{\eo}_{\pm}$ is elliptic of order
$2.$ Using~\eqref{7.21.10} in~\eqref{7.21.11} shows that
\begin{equation}
\cQ^{\mu}_{\pm}=\cQ^{0}_{\pm}+\left(\begin{matrix}
\sO_{-3} &\sO_{-2}\\ \sO_{-2} &\sO_{-3}\end{matrix}\right).
\label{7.21.12}
\end{equation}
We can now apply Proposition~\ref{prop6} to conclude that the $\sO_{-3}$ terms
along the diagonal in $\cQ^{\mu}_{\pm}$ can only change the symbol of
$\cP^{0}_{\pm}$ by terms with Heisenberg order $-4.$ The residue computations in
Section~\ref{s.3} show that the $\sO_{-2}$ off diagonal terms can only contribute
terms to $\cP^{\mu}_{\pm}$ at Heisenberg order $-2,$ hence
\begin{equation}
\cP^{\mu}_{\pm}=\cP^{0}_{\pm}+\left(\begin{matrix}\sO^{eH}_{-2,-4}
&\sO^{eH}_{-1,-2}\\
\sO^{eH}_{-1,-2}&\sO^{eH}_{-2,-4}\end{matrix}\right).
\label{7.21.13}
\end{equation}
The truth of~\eqref{7.21.9} is an immediate consequence of~\eqref{7.21.13} and
the fact that $\cU^0_{\pm}$ has extended Heisenberg orders $(0,1).$

As noted above, this shows that the leading order part of the parametrix for
$\cT^{\mu}_{\pm}$ has the form
\begin{equation}
\left(\begin{matrix} \cU^{\ooee}_{\pm} & \bzero \\
\bzero & \cU^{\eo}_{\pm}\end{matrix}\right).
\end{equation}
We let $\cU^{\mu}_{\pm}$ denote a parametrix chosen so that
\begin{equation}
\cU^{\mu}_{\pm}\cT^{\mu}_{\pm}=\Id+K_1^{\mu}\quad
\cT^{\mu}_{\pm}\cU^{\mu}_{\pm}=\Id+K_2^{\mu}
\end{equation}
with $K^{\mu}_1, K_2^{\mu}$ smoothing operators of finite rank. Arguing as in
Theorem~\ref{thm33}, one easily proves the desired regularity:
\begin{lemma}\label{lem14} Let $\mu\in\bbC$ and $s\geq 0,$ if $(f,g)$ belongs
  to $L^2,$ and satisfies
  \begin{equation}
    \label{eq:100.10}
     \cD^\mu_{\pm}\left(\begin{matrix}
      f\\ g\end{matrix}\right)=
\left(\begin{matrix} a\\b\end{matrix}\right)\in H^s\text{ and }
\cR_{\pm}\left(\begin{matrix}
      f\\g\end{matrix}\right)
\restrictedto_{bX}=0,
  \end{equation}
  in the sense of distributions, then $f,g\in H^{s+\frac 12}.$ There is a
  constant $C_{s,\mu},$ independent of $(f,g)$ so that
\begin{equation}
  \label{eq:100.11}
  \|(f,g)\|_{H^{s+\frac 12}}\leq C_{s,\mu}\left[\| \cD^\mu_{\pm}\left(\begin{matrix}
      f\\ g\end{matrix}\right)\|_{H^s}+\|\left(\begin{matrix}
      f\\ g\end{matrix}\right)\|_{L^2}\right].
\end{equation}
\end{lemma}
\begin{proof}
 Let $U_1=\cQ^{\mu}_{\pm}(a,b),$ so that $U_1\in H^{s+1},$ and
\begin{equation}
\cD^{\mu}_{\pm}(U-U_1)=0.
\end{equation}
On the one hand $\cR_{\pm}([U-U_1]_{bX})=-\cR_{\pm}([U_1]_{bX})\in H^{s+\frac 12}(bX).$ On the
other hand $[U-U_1]_{bX}\in\Im\cP^{\mu}_{\pm}$ and therefore
$$-\cR_{\pm}([U_1]_{bX})=\cR_{\pm}([U-U_1]_{bX})=\cT^{\mu}_{\pm}([U-U_1]_{bX}).$$
We apply $\cU^{\mu}_{\pm}$ to this relation to obtain
\begin{equation}
-\cU^{\mu}_{\pm}\cR_{\pm}([U_1]_{bX})=(\Id+K_1^{\mu})([U-U_1]_{bX})
\end{equation}
Rewriting this result gives
\begin{equation}
[U]_{bX}=-\cU^{\mu}_{\pm}\cR_{\pm}([U_1]_{bX})+(\Id+K_1^{\mu})([U_1]_{bX})
-K_1^{\mu}[U]_{bX}.
\label{7.22.12}
\end{equation}
All terms on the right hand side of~\eqref{7.22.12}, but the last are, by
construction, in $H^s(bX).$ Proposition~\ref{prop_bbw} implies that
$[U]_{bX}\in H^{-\ha},$ as $K_1^{\mu}$ is a smoothing operator, the last term,
$K_1^{\mu}[U]_{bX},$ is smooth. Thus $[U-U_1]_{bX}$ is in $H^s(bX),$ and
$U-U_1$ therefore belongs to $H^{s+\frac 12}(X);$ hence $U=U_1+U-U_1$ does as
well.  The estimate~\eqref{eq:100.11} follows easily from the definition of
$U_1$ and~\eqref{7.22.12}.
\end{proof}
Thus the solution $w$ constructed above is smooth on $\bX;$ this completes
the proofs of the proposition and Theorem~\ref{thm3}
\end{proof}

Using Theorem~\ref{thm3} we can describe the domains of
$\overline{(\eth^{\eo}_{\pm},\cR^{\prime\eo}_{\pm})}.$
\begin{corollary} The domains of the closures,
  $\overline{(\eth^{\eo}_{\pm},\cR^{\prime\eo}_{\pm})},$ are given by
\begin{equation}
\Dom(\overline{\eth^{\eo}_{\pm},\cR^{\prime\eo}_{\pm}})=
\{u\in L^2(X;F):\eth^{\eo}_{\pm}u\in L^2(X;F'),
\cR^{\prime\eo}_{\pm}u\restrictedto_{bX}=0
\}
\label{5.23.1}
\end{equation}
\end{corollary}
\begin{remark} Note that Proposition~\ref{prop_bbw} implies that
  $u\restrictedto_{bX}\in H^{-\ha}(bX).$ It is in this sense that the boundary
  condition in~\eqref{5.23.1} should be understood.
\end{remark}
\begin{proof} By Theorem~\ref{thm3}, we need only show that $u$ satisfying the
  conditions in~\eqref{5.23.1} belong to
  $\Dom((\eth^{\ooee}_{\pm},\cR^{\prime\ooee}_{\pm})^*).$ 
To show this we need only show that for \emph{smooth} sections, $v,$ with
  $\cR^{\ooee}_{\pm}v\restrictedto_{bX}=0,$ we have
\begin{equation}
\langle\eth^{\ooee}_{\pm} v,u\rangle=\langle v,\eth^{\eo}_{\pm} u\rangle.
\end{equation}
This follows by a simple limiting argument, because the map $\epsilon\mapsto
u\restrictedto_{\rho =\epsilon}$ is continuous in the $H^{-\ha}$-topology and
$v$ is smooth.
\end{proof}

As a corollary of Lemma~\ref{lem14} we get estimates for the second order operators
$\eth^{\ooee}_{\pm}\eth^{\eo}_{\pm},$ with subelliptic boundary conditions.
\begin{corollary}\label{cor3}  Let $X$ be a strictly pseudoconvex (pseudoconcave)
  manifold, $E\to X$ a compatible complex vector bundle. For $s\geq 0$ there
  exist constants $C_s$ such that if $u\in L^2, \eth^{\eo}_{E\pm}u\in L^2,
  \eth^{\ooee}_{E\pm}\eth^{\eo}_{E\pm}u\in H^s$ and $\cR^{\prime\eo}_{E\pm} [u]_{bX}=0,
  \cR^{\prime\ooee}_{E\pm}[\eth^{\eo}_{E\pm}u]=0$ in the sense of distributions, then
\begin{equation}
\|u\|_{H^{s+1}}\leq
C_s[\|\eth^{\ooee}_{E\pm}\eth^{\eo}_{E\pm}u\|_{H^s}+\|u\|_{L^2}].
\label{7.23.3}
\end{equation}
\end{corollary}
\begin{proof}
We apply Lemma~\ref{lem14} to $U=(u,\eth^{\eo}_{\pm} u).$ Initially we
see that $\cD^{0}_{\pm} U\in L^2.$ The lemma shows that $\eth^{\eo}_{\pm}u\in H^{\ha},$
and therefore $\cD^{0}_{\pm} U\in H^{\ha}.$ Applying the lemma
recursively, we eventually deduce that $\cD^{0}_{\pm} U\in H^{s}$ and that
there is constant $C_s'$ so that
\begin{equation}
\|u\|_{H^{s+\ha}}+\|\eth^{\eo}_{\pm} u\|_{H^{s+\ha}}\leq
C_s'[\|\eth^{\ooee}_{\pm}\eth^{\eo}_{\pm}u\|_{H^s}+\|u\|_{L^2}].
\end{equation}
It follows from Theorem~\ref{thm33} that, for a constant $C_s'',$ we have
\begin{equation}
\|u\|_{H^{s+1}}\leq C_s''[\|u\|_{H^{s+\ha}}+\|\eth^{\eo}_{\pm} u\|_{H^{s+\ha}}]
\end{equation}
Combining the two estimates gives~\eqref{7.23.3}.
\end{proof}

In the case that $X$ is a complex manifold with boundary, these estimates imply
analogous results for the modified $\dbar$-Neumann problem acting on individual
form degrees. These results are stated and deduced from Corollary~\ref{cor3}
in~\cite{Epstein4}.


\begin{thebibliography}{10}

\bibitem{Beals-Greiner1}
{\sc R.~Beals and P.~Greiner}, {\em Calculus on {H}eisenberg {M}anifolds},
  vol.~119 of Annals of Mathematics Studies, Princeton University Press, 1988.

\bibitem{Beals-Stanton}
{\sc R.~Beals and N.~Stanton}, {\em The heat equation for the $\dbar$-{N}eumann
  problem, {I}}, Comm. PDE, 12 (1987), pp.~351--413.

\bibitem{BBW}
{\sc B.~Booss-Bavnbek and K.~P. Wojciechowsi}, {\em Elliptic Boundary Problems
  for the Dirac Operator}, Birkh\"auser, Boston, 1996.

\bibitem{duistermaat}
{\sc J.~Duistermaat}, {\em The heat kernel Lefschetz fixed point theorem
  formula for the Spin-c Dirac Operator}, Birkh\"auser, Boston, 1996.

\bibitem{Epstein4}
{\sc C.~L. Epstein}, {\em Subelliptic {S}pin${}_{\bbC}$ {D}irac operators, {I}},
  Annals of Math.,  166 (2007), pp.~183-214.

\bibitem{EpsteinMelrose3}
{\sc C.~L. Epstein and R.~Melrose}, {\em The {H}eisenberg algebra, index theory
  and homology}, preprint, 2004.

\bibitem{Greiner-Stein1}
{\sc P.~C. Greiner and E.~M. Stein}, {\em Estimates for the
  {$\overline{\partial}$}-{N}eumann problem}, Mathematical Notes, Princeton
  University Press, 1977.

\bibitem{Hormander3}
{\sc L.~H{\"o}rmander}, {\em The Analysis of Linear Partial Differential
  Operators}, vol.~3, Springer Verlag, Berlin{,} Heidelberg{,} New York{,}
  Tokyo, 1985.

\bibitem{Kato}
{\sc T.~Kato}, {\em Perturbation Theory for Linear Operators, corrected 2nd
  printing} Springer Verlag, Berlin Heidelberg, 1980.


\bibitem{KobayashiNomizu2}
{\sc S.~Kobayashi and K.~Nomizu}, {\em Foundations of Differential Geometry,
  Vol. 2}, John Wiley and Sons, New York, 1969.

\bibitem{LawsonMichelsohn}
{\sc H.~B. {Lawson Jr.} and M.-L. Michelsohn}, {\em Spin Geometry}, vol.~38 of
  Princeton Mathematical Series, Princeton University Press, 1989.

\bibitem{Seeley0}
{\sc R.~Seeley}, {\em Singular integrals and boundary value problems}, Amer.
  Jour. of Math., 88 (1966), pp.~781--809.

\bibitem{Taylor3}
{\sc M.~E. Taylor}, {\em Noncommutative microlocal analysis, part {I}},
  vol.~313 of Mem. Amer. Math. Soc., AMS, 1984.

\bibitem{Taylor6}
\leavevmode\vrule height 2pt depth -1.6pt width 23pt, {\em Partial Differential
  Equations, Vol. 2}, vol.~116 of Applied Mathematical Sciences, Springer Verlag, New
  York, 1996.

\bibitem{wells}
{\sc R.~Wells}, {\em Differential Analysis on Complex Manifolds}, vol.~165 of
  Graduate Texts in Mathematics, Springer Verlag, New York, 1980.

\end{thebibliography}
\end{document}